\documentclass[reqno]{amsart}
\usepackage[T1]{fontenc}
\usepackage[latin1]{inputenc}
\usepackage[english]{babel}
\usepackage[babel]{csquotes}

\usepackage{cite}
\usepackage{amssymb}
\usepackage{amsmath}
\usepackage{amsthm}
\usepackage{latexsym}
\usepackage{graphicx}
\usepackage{mathrsfs}
\usepackage{mathrsfs}
\usepackage{bbm}
\usepackage{verbatim}

\usepackage[usenames,dvipsnames]{color}

\newtheorem{thm}{Theorem}[section]

\newtheorem{lem}[thm]{Lemma}
\newtheorem{prop}[thm]{Proposition}

\def\enne{\mathbb{N}}
\def\zeta{\mathbb{Z}}

\def\erre{\mathbb{R}}
\def\Rz{\mathbb{R}}

\def\P{\mathbb{P}}

\def\E{\mathop{{}\mathbb{E}}}

\def\cL{\mathscr{L}}
\def\cF{\mathscr{F}}
\def\cB{\mathscr{B}}
\def\eps{\varepsilon}
\def\cP{\mathscr{P}}

\def\OO{\mathcal{O}}
\def\U{\mathcal{U}}

\def\beq{\begin{equation}}
\def\eeq{\end{equation}}

\def\to{\rightarrow}
\def\wto{\rightharpoonup}
\def\wstarto{\stackrel{*}{\rightharpoonup}}
\def\embed{\hookrightarrow}
\def\cembed{\stackrel{c}{\hookrightarrow}}

\def\norm #1{\left\|#1\right\|}

\def\sp #1#2{\left<#1,#2\right>}
\newcommand\ip\sp

\renewcommand{\d}{{\mathrm d}}

\begin{document}
\title[Stochastic PDEs via convex minimization]
{Stochastic PDEs via convex minimization}

\author{Luca Scarpa}
\address[Luca Scarpa]{Faculty of Mathematics, University of Vienna, 
Oskar-Morgenstern-Platz 1, 1090 Wien, Austria.}
\email{luca.scarpa@univie.ac.at}
\urladdr{http://www.mat.univie.ac.at/$\sim$scarpa}

\author{Ulisse Stefanelli}
\address[Ulisse Stefanelli]{Faculty of Mathematics, University of Vienna, 
Oskar-Morgenstern-Platz 1, 1090 Wien, Austria, Vienna Research Platform on Accelerating
  Photoreaction Discovery, University of Vienna, W\"ahringerstra\ss e 17, 1090 Wien, Austria, and  Istituto di Matematica
Applicata e Tecnologie Informatiche \textit{{E. Magenes}}, v. Ferrata 1, 27100
Pavia, Italy.}
\email{ulisse.stefanelli@univie.ac.at}
\urladdr{http://www.mat.univie.ac.at/$\sim$stefanelli}



\begin{abstract}
We prove the applicability of the Weighted Energy-Dissipation
(WED)
variational principle \cite{ms3} to
nonlinear parabolic stochastic partial
differential equations in abstract form. The WED principle consists in the
minimization of a parameter-dependent convex 
functional on entire trajectories. Its unique minimizers correspond to
elliptic-in-time regularizations of the stochastic differential
problem. As the regularization parameter tends to zero, solutions of
the limiting problem are recovered.   This in particular provides a direct approch via
convex optimization to the approximation of nonlinear stochastic partial
differential equations.
\end{abstract}

\maketitle


\section{Introduction}
\setcounter{equation}{0}
\label{sec:intro}

This paper is concerned with stochastic quasilinear partial differential equations of the form 
\begin{equation}
\label{eq:00} 
\d u - {\rm div}\,( D\phi(t,\nabla u)) \,\d t  
 + D \psi (t,u)\, \d t\ni f(t) \, \d t+
  B(t) \, \d W\,,
\end{equation}
complemented with suitable boundary and initial conditions.
Here, the real-valued function $u$ is defined on $\Omega \times
[0,T] \times {\mathcal O}$, where $(\Omega,\cF, \P)$ is a probability space,  
$ {\mathcal O} \subset \Rz^d$ is a smooth bounded domain, 
and $T>0$ is a reference time. The functions
$\phi(t,\cdot): \Rz^d \to \Rz$ and $\psi(t,\cdot): \Rz \to \Rz$ are
asked to be convex, the
gradients $D\phi$ and $D\psi$ are taken with respect to the second
variable only,
and the time-dependent sources $f$ and $B$ are
given. In particular, $B(\cdot)\in \cL^2(U;L^2(\mathcal O))$
(Hilbert-Schmidt operators) is 
stochastically integrable with respect to $W$, a
cylindrical Wiener process on a separable Hilbert space $U$. 

Under different choices for the nonlinearities $\phi$ and $\psi$,
equation \eqref{eq:00} may arise in connection with various classical
models, including the Allen-Cahn and the $p$-Laplace equation. 
Assume equation \eqref{eq:00} to be complemented with homogeneous
Dirichlet boundary conditions, for notational simplicity, and with the
initial condition $u(0)=u_0$, where $u_0$ is some suitable
initial datum. 
Letting $\phi(t,\cdot)$ and $\psi(t,\cdot)$ be of $p$-growth, equation \eqref{eq:00} can be
weakly formulated in the dual of the space $W^{1,p}_0(\mathcal O)$,
according to the classical theory by Pardoux~\cite{Pard} and 
Krylov--Rozovski\u\i~\cite{KR-SPDE}. It is well-known  that the
solution $u$ is an It\^o process, in the sense that it can be represented in the general form 
\begin{equation}\label{con}
u = u^d + \int_0^\cdot u^s(r)\,\d W(r)\,,
\end{equation}
where the process $u^d$ is differentiable in time and $u^s$
is $\cL^2(U;L^2(\mathcal O))$-valued and 
stochastically integrable with respect to $W$.
This decomposition into the deterministic part $u^d$ and the
stochastic part $u^s$ is unique. With this notation, 
$u$ is a solution to
the original problem \eqref{eq:00} 
if and only if $u$ satisfies the constraint \eqref{con} and the equations
\begin{align*}
  \partial_t u^d - {\rm div}\,( D\phi(\cdot,\nabla u))
   + D \psi (\cdot,u) \ni f\,, \qquad u^s=B\,, \qquad u^d(0)=u_0\,.
\end{align*}

The aim of this paper is to tackle the weak solvability of equation
\eqref{eq:00} via the Weighted Energy-Dissipation
(WED) variational approach. This hinges upon the minimization of the
parameter-dependent functional $I_\eps$ on entire trajectories, the
so-called WED functional, 
given~by
\begin{align*}
  I_\eps(u) &=
  \E\displaystyle\int_0^T\!\!\int_{\mathcal O} e^{{-r}/{\eps}}
  \left(\frac\eps2|\partial_t u^d(r)|^2 + \phi(r,\nabla u(r)) +
    \psi(r,u(r)) - f(r)\, u(r)
 \right)\,\d x\,\d r\\
 &\quad+ \E\displaystyle\int_0^T 
 e^{{-r}/{\eps}}\frac12\norm{u^s(r)-B(r)}^2_{\cL^2(U,L^2(\mathcal
   O))}\, \d r\,.
\end{align*}
The convex WED functional $I_\eps$ has to be minimized under two linear
constraints, namely the decomposition \eqref{con} and the initial
condition $u(0)=u_0$. This results in a {\it
  convex minimization} problem. Our main result, Theorem \ref{thm:00},
states that that, under
suitable assumptions on data, 
\begin{center}
  \begin{minipage}{0.82\linewidth}
    for all $\eps>0$ the minimizer $u_\eps$ of $I_\eps$ uniquely
    exists. As $\eps \to 0$ we have that $u_\eps\to u$ where $u$ is the unique solution of 
    the stochastic differential problem~\eqref{eq:00}.
  \end{minipage}
\end{center}
This provides a new variational approximation to
the stochastic differential problem \eqref{eq:00}, making it
accessible to a direct optimization approach, and paving the way to the
application of the far-reaching tools of
the calculus of variations \cite{Buttazzo,Dacorogna,DalMaso}.

The role of the exponential weight in $I_\eps$ is revealed by
computing the corresponding Euler-Lagrange equation. In the
current setting these formally read  
\begin{align*}
   &-\eps\partial_t(\partial_t u^d_\eps)^d + \partial_t u^d_\eps -
   {\rm div}\,( D\phi(\cdot,\nabla u_\eps)) + D\psi (\cdot,u_\eps)=f\,,
   \qquad 
    u_\eps^s=B + \eps(\partial_t u^d_\eps)^s ,\\
  &\eps\partial_tu^d_\eps(T)=0\,, \qquad u_\eps^d(0)=u_{0}\,,
\end{align*}
where we have also included the initial condition, for completeness.
In particular, the minimizers $u_\eps$ solve an
\emph{elliptic-in-time regularization} of the stochastic differential
problem \eqref{eq:00}, complemented by an extra Neumann boundary
condition at $T$. Note that for all $\eps>0$ the problem is {\it
not causal} and that causality is restored in the limit $\eps \to 0$.
 
Elliptic-regularization techniques for nonlinear PDEs are quite
classical. Introduced by  Lions in \cite{Lions:63}, they have
been used  
by Kohn \& Nirenberg \cite{Kohn65},  Olein\u{i}k
\cite{Oleinik}, and again Lions \cite{Lions:63a,Lions:65} in order to
investigate regularity. An account on linear results can be found the 
 the book by {Lions \& Magenes} 
 \cite{Lions-Magenes1}, whereas an early result on solvability in a
 nonlinear setting is due to Barbu \cite{Barbu75}.

The variational formulation of elliptic-regularization via WED
functionals can be traced back to Ilmanen \cite{Ilmanen}, who used it in the context of Brakke mean-curvature
flow of varifolds, and to Hirano \cite{Hirano94} in connection with periodic
solutions of gradient flows. A reference to 
 WED functionals is already 
pointed out in the classical textbook by  Evans
\cite[Problem 3, p.~487]{Evans98}.

The WED variational approach has been applied to a variety of
different parabolic problems, including gradient flows
\cite{akno,BDM15,CO08,ms3},
rate-independent flows \cite{MiOr2008,ms2}, crack propagation
\cite{Larsen-et-al09}, doubly-nonlinear flows
\cite{akmest,AKSt2010,AkSt2011,AkSt3,Me2}, nonpotential perturbations
\cite{akme,Me} and variational approximations \cite{LiMe}, curves of maximal slope
in metric spaces \cite{RoSaSeSt,RoSaSeSt2,segatti}, mean curvature flow
\cite{Ilmanen,SpSt}, dynamic plasticity \cite{Davoli}, and the incompressible
Navier-Stokes system \cite{OSS}. 
 
Motivated by a conjecture by De Giorgi
\cite{Degiorgi96}, the WED variational approach has been extended to
semilinear wave equations \cite{SeTi,St2011}. Extensions to other
classes of hyperbolic problems including mixed
hyperbolic-parabolic equations
\cite{LiSt,LiSt2,Serra-Tilli14,tenta1} and nonhomogeneous equations
\cite{tenta2,tenta3} have also been addressed.

In the context of stochastic PDEs, the application of tools 
from calculus of variations in order to characterize variational solutions 
is much less developed, and
has been employed so far mainly in connection with the Brezis-Ekeland principle. 
In this direction, we mention the pioneering works by Barbu \& R\"ockner
\cite{barbu-opt, barbu-roc} dealing with SPDEs with additive 
and linear multiplicative noise, and by Krylov \cite{kry-ext}.
More recently, Boroushaki \& Ghoussoub \cite{bor-ghou}
generalized these results also to the case of multiplicative noise,
by characterizing solutions as minima of self-dual functionals.

This paper contributes to the first application of the WED principle in the stochastic
setting. Compared with the deterministic situation, the theory is here
much more involved. 

The first main difficulty arises in proving existence of minimizers
for $I_\eps$. This requires the characterization of the
subdifferential of $I_\eps$ in terms of the Euler-Lagrange problem.
In the stochastic setting, this $\eps$-regularized problem 
consist of a forward-backward system of SPDEs.
The identification of the Euler-Lagrange equation is more
delicate
compared to the deterministic framework.
In the deterministic case, it is well known that the 
space of compactly-supported $C^k$ test-functions $C^k_c(0,T)$
is dense in $L^2(0,T)$ for all $k\in\enne$: this allows to identify the Euler-Lagrange 
equation pretty straightforwardly at least in a weak sense.
 By contrast,
due to the presence of 
nonzero martingales in $L^2(\Omega; L^2(0,T))$,
the usual deterministic techniques do not apply here,
 and the Euler-Lagrange equation has to be characterized 
using different tools, both on the analytical side and
the probabilistic side. As a matter of fact, on the one hand we 
need to introduce suitable functional spaces of processes in Banach spaces
(It\^o processes), and on the other hand we rely on 
the abstract variational theory for backward SPDEs and 
martingale representation theorems
in infinite dimensional spaces. 

The second main difficulty concerns 
proving the well-posedness of the
Euler-Lagrange problem. As we have pointed out above, 
the second-order Euler-Lagrange equation is noncausal and 
corresponds to a system of a forward and
a backward first-order stochastic equation. 
The discussion of this forward-backward system calls for a further approximation
on the nonlinearity. Identifications of nonlinear limits are performed via lower
semicontinuity arguments, which in turn rely on
specific It\^o's formulas, both at the approximate and at the limit level. 

In the paper, we actually consider a general class of abstract
equations, including \eqref{eq:00}. Indeed, 
we frame the problem in the abstract variational setting
of a 
Gelfand triple $(V,H,V^*)$ and focus on 
\[
  \d u + A(t,u)\,\d t \ni B\,\d W\,, \qquad u(0)=u_0\,,
\]
where $A$ is a time-dependent subdifferential-type operator from $V$ to $V^*$,
$V$ being a separable reflexive Banach space and $H$ a separable Hilbert space.
We collect all relevant notation, list assumptions, and state Theorem
\ref{thm:00}, our main
result, in Section \ref{sec:main}. The proof of
Theorem \ref{thm:00} is then split into Section \ref{sec:thm1}
(Euler-Lagrange problem), Section \ref{sec:thm2} (convergence as
$\eps \to 0$), and Section \ref{sec:equiv} (existence of minimizers).


\section{Main result}
\label{sec:main}
 
In the following, we directly focus on the abstract Cauchy problem 
\beq\label{eq:0}
  \d u + \partial \Phi(t,u)\,\d t \ni B\,\d W\,, \qquad u(0)=u_0\,.
\eeq
The latter arises as variational
formulation of an initial and boundary value problem for equation
\eqref{eq:00} by choosing the convex map $\Phi(t,\cdot)$ as
$$\Phi_p(t,\cdot): W^{1,p}_0(\mathcal O) \to (-\infty,\infty], \quad \Phi_p(t,u):=
  \int_{\mathcal O}\left(\phi(t,\nabla u) + \psi(t,u)\right) \, \d x\,.
$$
Note that we have neglected the deterministic forcing $f$ in
\eqref{eq:00} for the sake of
notational simplicity. Indeed, this could be included in the analysis
with no specific difficulty. 

In this section we introduce the necessary notation and assumptions to make the
meaning of problem \eqref{eq:0} precise and we state of our main result, Theorem \ref{thm:00}. This
is then proved in Sections~\ref{sec:thm1}-\ref{sec:equiv}.

Let $(\Omega,\cF,\P)$ be a probability space endowed with a
complete and right-continuous filtration 
$(\cF_t)_{t\in[0,T]}$, where $T>0$ is a fixed final time. Let also 
$W$ be a cylindrical Wiener process on a separable Hilbert space $U$.
We will assume that $(\cF_t)_{t\in[0,T]}$ is the natural 
augmented filtration associated to $W$.
The progressive $\sigma$-algebra on $\Omega\times[0,T]$ will be 
denoted by $\cP$. For any Banach space $E$, 
the norm in $E$ will be denoted by $\norm{\cdot}_E$.
 For any $r,s\in\mathopen[1,+\infty\mathclose)$ we denote by 
$L^r_\cP(\Omega; L^s(0,T; E))$ the usual space of 
Bochner-integrable functions which are strongly $\cP$-measurable 
from $\Omega\times[0,T]$ to $E$.
When $r>1$ and $s=+\infty$, we explicitly define
\[
  L^r_\cP(\Omega; L^\infty(0,T; E^*)):=
  \left\{v:\Omega\to L^\infty(0,T; E^*) \text{ weakly* meas.}\,:\,
  \E\norm{v}_{L^\infty(0,T; E^*)}^r<\infty
  \right\}
\]
where for any $f\in L^1(\Omega)$ we use the standard
notation $\E f : = \int_\Omega f \,\d \P$ for the expected value. 
Recall that 
by \cite[Thm.~8.20.3]{edwards} we have the identification
\[
L^r_\cP(\Omega; L^\infty(0,T; E^*))=
\left(L^{r/(r-1)}_\cP(\Omega; L^1(0,T; E))\right)^*\,.
\]
Moreover, for any $r\geq1$, the
symbol $L^r(\Omega; C^0([0,T]; E))$ denotes the space
of $r$-integrable continuous adapted process
(hence also progressively measurable) with values in $E$.
For any pair of separable Hilbert spaces $E_1$ and $E_2$, 
we will use the symbols $\cL(E_1,E_2)$ and $\cL^2(E_1,E_2)$
for the spaces of linear continuous and 
Hilbert-Schmidt operators from $E_1$ and $E_2$, respectively.

Let us fix now a useful notation 
in order to denote suitable spaces of It\^o processes.
For every separable reflexive Banach space $E_1$ and any
Hilbert spaces $E$, $E_2$,
with $E_1,E_2\embed E$ continuously,
and for any $s,r\in[1,+\infty)$, we use the notation
\[
  \mathcal I^{s,r}(E_1,E_2):=
  L^s_\cP(\Omega; W^{1,s}(0,T; E_1))\oplus
  \left[L^r_{\cP}(\Omega; L^2(0,T; \cL^2(U, E_2)))\cdot W\right]\,,
\]
where we have used the classical symbol $\cdot W$ to denote
stochastic integration with respect to $W$.
Equivalently, we have the representation 
\begin{align*}
  \mathcal I^{s,r}(E_1,E_2)&=\left\{z=z^d + z^s\cdot W:\right.\\
  &\qquad\left.z^d\in L^s_\cP(\Omega; W^{1,s}(0,T; E_1))\,,\quad
  z^s\in L^r_{\cP}(\Omega; L^2(0,T; \cL^2(U, E_2))) \right\}\,.
\end{align*}
The latter specifies that the two components $z^d$ and $z^s$ are uniquely 
determined from the process $z$, see \eqref{con},
 so that the sum appearing above is actually 
a direct sum, and the projections 
\begin{align*}
  &\Pi^d:\mathcal I^{s,r}(E_1,E_2)\to L^s_\cP(\Omega; W^{1,s}(0,T; E_1))\,, \quad z\mapsto z^d\,,\\
  &\Pi^s:\mathcal I^{s,r}(E_1,E_2)\to L^r_\cP(\Omega; L^2(0,T; \cL^2(U,E_2)))\,, \quad z\mapsto z^s\,,
\end{align*}
are well-defined, linear, and continuous. 
Let us also point out that the space $\mathcal I^{s,r}(E_1,E_2)$ is a 
Banach space,
and even a Hilbert space if $s=r=2$ and $E_1$ is a Hilbert space.
A natural norm on $\mathcal I^{s,r}(E_1,E_2)$ is given by 
\[
  \norm{z}_{\mathcal I^{s,r}(E_1,E_2)}:=
  \norm{z^d}_{L^s_\cP(\Omega; W^{1,s}(0,T; E_1))}+
  \norm{z^s}_{L^r_{\cP}(\Omega; L^2(0,T; \cL^2(U, E_2)))}\,, \qquad z\in\mathcal I^{s,r}(E_1,E_2)\,.
\]

Throughout the paper, we assume the following setting. 
\begin{description}
  \item[H0]  $H$ and $V_0$ are separable Hilbert spaces and 
$V$ is a separable reflexive Banach space,
with $V_0 \embed  V\embed H$ continuously and densely. 
\end{description}
In the sequel, we will
identify $H$ with its dual $H^*$ in the canonical way, so that 
we have the continuous and dense inclusions 
\[
V_0\embed V\embed H\embed V^*\embed V_0^*\,.
\]
The scalar product in $H$ and the duality pairing 
between $V^*$ and $V$
(and between $V_0^*$ and $V_0$)
will be denoted by the symbols $(\cdot, \cdot)$ and $\ip{\cdot}{\cdot}$, respectively.

We assume the following hypotheses.
\begin{description}
  \item[H1] $\Phi:\Omega\times[0,T]\times V\to [0,+\infty)$ is $\cP\otimes\cB(V)$-measurable,
  and $\Phi(\omega,t,\cdot):V\to[0,+\infty)$ is convex and 
  lower semicontinuous. We let 
  $A(\omega,t,\cdot):=\partial \Phi(\omega,t,\cdot):V\to 2^V$ for almost every
  $(\omega,t)\in\Omega\times[0,T]$.
  Moreover, we ask for
  constants $c_A, C_A>0$ and $p\in\mathopen[2,+\infty\mathclose)$, and a 
  $\cP$-measurable process $f_A\in L^1(\Omega\times(0,T))$
  such that, setting $q:=p/(p-1)$,
  \[
    \ip{v}{z}\geq c_a\norm{z}_V^p\,, \qquad\qquad
    \norm{v}_{V^*}^q\leq f_A(\omega,t) + C_A\norm{z}_V^p\,,
  \]
  for almost every $(\omega,t)\in\Omega\times[0,T]$, for every 
  $z\in V$, and for every $v\in A(\omega,t,z)$.
  \item[H2] $u_0\in L^2(\Omega, \cF_0; H)$ and $B\in L^2_{\cP}(\Omega; L^2(0,T; \cL^2(U,H)))$.
\end{description}
Let us point out that the progressive measurability of $\Phi$
required in {\bf H1}
implies that the operator
$A$ is $\cP\otimes\cB(V)/\cB(V^*)$-Effros-measurable,
in the sense of
\cite{Hess:meas,Rock:int}.

 Before moving on, let us comment on the choice of the space $V_0$.
  The introduction of $V_0$ will be needed in the paper
  since at some point we would have to rely on It\^o's formula for the square 
  of the $V^*$-norm. However, this cannot be done in general 
  if $V$ is a Banach space: indeed, in such case
  the duality mapping of $V^*$ is nonlinear and possibly 
  not twice Fr\'echet-differentiable, hence the required It\^o formula 
  is not trivial and not known in general, even in the extended 
  framework of stochastic integration in UMD Banach spaces
  (see \cite{bra-vanneer, van-neer, van-neer2}). 
  The introduction of the space $V_0$ is then employed to 
  bypass this problem by exploiting
  its structure as Hilbert space,
  and allows to write an It\^o formula in $V_0^*$.
  Clearly, if $V$ is a Hilbert space itself, the optimal choice of $V_0$ is given by 
  $V_0=V$. In general, if $V$ is only a Banach space, roughly speaking
  one should ideally choose the space $V_0$ {\it as large as possble}.
  For example, if $V=W^{s,\ell}(\OO)$ for a certain domain $\OO\subset\erre^d$
  with Lipschitz boundary, with $\ell\in(2,+\infty)$ and $s>0$, one could choose
  \[
  V_0=H^{s'}(\OO)\,, \qquad \forall\,s' \geq s+\frac{d}2-\frac{d}{\ell}\,,
  \]
  with the choice $s' = s+{d}/2-{d}/{\ell}$ being optimal in this sense.

The classical variational theory on SPDEs
(see \cite{Pard, KR-SPDE}) ensures that under the assumptions
{\bf H0}--{\bf H2} the Cauchy problem \eqref{eq:0} admits 
a unique solution $(u,\xi)$, with 
\beq\label{1}
  u\in L^2(\Omega; C^0([0,T]; H))\cap L^p_\cP(\Omega; L^p(0,T; V))\qquad
  \xi\in L^q_\cP(\Omega; L^q(0,T; V^*))\,,
\eeq
such that 
\beq\label{2}
  \xi(\omega,t) \in A(\omega,t,u(\omega,t)) \quad\text{for a.e.~}(\omega,t)\in\Omega\times(0,T)\,,
\eeq
and
\beq\label{3}
  u(t) + \int_0^t\xi(s)\,\d s = u_0 + \int_0^tB(s)\,\d W(s) \quad\text{in } V^*\,, \quad
  \forall\,t\in[0,T]\,,\quad\P\text{-a.s.}
\eeq

Let us reformulate this solution concept in a different fashion.
We introduce the space
\begin{align*}
  \U&:=\left\{z\in
  L^2(\Omega; C^0([0,T]; H))\cap L^p_\cP(\Omega; L^p(0,T; V)):\;z=z^d+z^s\cdot W\,,\right.\\
  &\qquad\left.
  z^d\in L^q_\cP(\Omega; W^{1,q}(0,T; V^*))\,, \quad 
  z^s\in L^2_\cP(\Omega; L^2(0,T; \cL^2(U,H)))\right\}\,.
\end{align*}
Note that $\mathcal U$ can be written in compact form as
\begin{align*}
\mathcal U&=
L^2(\Omega; C^0([0,T]; H))\cap L^p_\cP(\Omega; L^p(0,T; V)) \cap \mathcal I^{q,2}(V^*,H)\,,
\end{align*}
so that in particular $\mathcal U$ is a Banach space.

With this notation,
the process $u$ solves the problem \eqref{1}--\eqref{3} if and only if
\beq\label{prob}
  u \in \U\,, \qquad 
  \begin{cases}
  \partial_t u^d + A(u) \ni 0\,,\\
  u^d(0)=u_0\,,\\
  u^s=B\,.
  \end{cases}
\eeq
In such a case, \eqref{1}--\eqref{3} are satisfied with the choice
$\xi:=-\partial_t u^d$.

As mentioned, the WED approach consists in minimizing an
$\eps$-dependent
functional over entire trajectories and passing to the limit in the
parameter $\eps$. This procedure results in an elliptic regularization 
in time, hence delivering regular approximations. In particular, the
differential problem \eqref{eq:0} is reformulated as a linearly constrained
convex minimization. In the abstract setting of \eqref{1}-\eqref{3}, letting $\eps>0$
we introduce the WED functional 
\[
I_\eps:L^p_\cP(\Omega; L^p(0,T; V))\cap \mathcal I^{2,2}(H,H)=:\mathcal V\to [0,+\infty]\,,
\]
as
\begin{align*}
  I_\eps(u):=
\left\{
    \begin{array}{ll}
  \E\displaystyle\int_0^Te^{{-s}/{\eps}}
  \left[\frac\eps2\norm{\partial_t u^d(s)}_H^2 + \Phi(s,u(s))
  + \frac12\norm{(u^s- B)(s)}^2_{\cL^2(U,H)}\right]\,\d s \\[3mm]
   \qquad\text{if } u^d(0)=u_{0,\eps}\,,\\
 +\infty\\ 
 \qquad\text{if } u^d(0)\neq u_{0,\eps}\,.
    \end{array}
\right.
\end{align*}
We qualify the $\eps$-dependent  initial data $(u_{0,\eps})_\eps$ 
above by requiring that 
the sequence
\beq\label{data:eps}
  (u_{0,\eps})_{\eps>0}\subset L^p(\Omega,\cF_0;V_0)
\eeq
is well-prepared, in the sense that, as $\eps\searrow0$,
\begin{align}
\label{ip_init1}
&u_{0,\eps}\to u_0 \quad\text{in } L^2(\Omega, \cF_0; H)\,, 
&&\eps \norm{u_{0,\eps}}^p_{L^p(\Omega, \cF_0 ; V_0)}\to 0\,.
\end{align}
The existence of sequences fulfilling \eqref{data:eps}--\eqref{ip_init1} 
follows directly from {\bf H2}
and the density of $V_0\embed H$, by
standard regularization techniques. 

Minimizers of $I_\eps$ will be proved to belong to the space

\begin{align*}
  \U_{reg}&:=\left\{z\in
  L^2(\Omega; C^0([0,T]; H))\cap L^p_\cP(\Omega; L^p(0,T; V)):\;z=z^d+z^s\cdot W\,,\right.\\
  &\qquad\left.
  z^d\in L^q_\cP(\Omega; C^1([0,T]; V_0^*))\cap L^2_\cP(\Omega; H^1(0,T; H))\right.\,,\\
  &\qquad\left.z^s\in L^2_\cP(\Omega; L^2(0,T; \cL^2(U,H)))\,,\right.\\
  &\qquad\left.\partial_t z^d=(\partial_t z^d)^d + (\partial_t z^d)^s\cdot W\,,\right.\\
  &\qquad\left.(\partial_t z^d)^d\in L^q_\cP(\Omega; W^{1,q}(0,T; V^*))\,,
  \quad(\partial_t z^d)^s\in L^2_\cP(\Omega; L^2(0,T; \cL^2(U,H)))\right\}\,.
\end{align*}
Again, note that a more compact notation for $\mathcal U_{reg}$ reads
\[
  \mathcal U_{reg}=\{z\in L^p_\cP(\Omega; L^p(0,T; V))
  \cap \mathcal I^{2,2}(H,H):
  \partial_t z^d\in\mathcal I^{q,2}(V^*, H)\}\,.
\]
Let us point out in particular that $\mathcal U_{reg}\embed \mathcal V\embed \mathcal U$
with continuous inclusions.

The Euler-Lagrange equation for functional $I_\eps$ corresponds
to the $\eps$-regularized problem 
\beq
  \label{prob_eps}
  u_\eps\in\U_{reg}\,, \qquad
  \begin{cases}
    -\eps\partial_t(\partial_t u^d_\eps)^d + \partial_t u^d_\eps + A(u_\eps) \ni 0\,,\\
    \eps\partial_tu^d_\eps(T)=0\,,\\
    u_\eps^d(0)=u_{0,\eps}\,,\\
     u_\eps^s=B + \eps(\partial_t u^d_\eps)^s\,.
  \end{cases}
\eeq
Note that the second-order problem \eqref{prob_eps} can be seen as a system 
of two equations of first order in time, one forward and one backward, 
by using the classical substitution $v_\eps:=\partial_tu^d_\eps$.
Indeed, with this notation \eqref{prob_eps} is equivalent to 
\beq
  \label{prob_eps_bis}
  \begin{cases}
    \d u_\eps = v_\eps\,\d t + (B+\eps G_\eps)\,\d W\,,\\
    u_\eps(0)=u_{0,\eps}\,,
  \end{cases}
  \qquad
  \begin{cases}
    -\eps\d v_\eps + v_\eps\,\d t + A(u_\eps)\,\d t \ni -\eps G_\eps\,\d W\,,\\
    \eps v_\eps(T)=0\,.
  \end{cases}
\eeq
Note that the variables of the forward-backward system
\eqref{prob_eps_bis} are three,
namely $u_\eps$, $v_\eps$, and $G_\eps$. Indeed, while
the forward equation has a unique variable ($u_\eps$), 
the concept of solution for the backward stochastic equation requires
the two variables $v_\eps$ and $G_\eps$
due 
to the need of representation theorems for martingales.
In particular, we have that 
$G_\eps=v_\eps^s$ is uniquely determined by the backward stochastic equation.

The main result of the paper reads as follows.

\begin{thm}[Weighted Energy-Dissipation approach]\label{thm:00} 
Assume \emph{\bf H0--H2}. Then:
\begin{itemize}
\item[i)] \emph{(Minimization)} For all $\eps>0$ the functional
$I_\eps$ admits a unique global mimizer $u_\eps \in \mathcal
V$.\medskip
\item[ii)] \emph{(Euler-Lagrange equation)} 
The minimizer also satisfies $u_\eps \in \mathcal U_{reg}$ and it is the unique solution
to the problem \eqref{prob_eps}. Namely, there exists a unique
triplet $(\xi_\eps,v_\eps ,G_\eps)$ with
\begin{align*}
  &\xi_\eps\in L^q_\cP(\Omega; L^q(0,T; V^*))\,,\\
  &v_\eps\in L^q(\Omega; C^0([0,T]; V_0^*))\cap L^2_\cP(\Omega; L^2(0,T; H))\,,\\
  &G_\eps\in L^q_\cP(\Omega; L^2(0,T; \cL^2(U,V_0^*)))\,,
  \end{align*}
 such that
  \begin{align*}
  \xi_\eps(\omega,t) \in A(\omega,t,u_\eps(\omega,t)) \quad\text{for a.e.~}(\omega,t)\in\Omega\times(0,T)\,,
  \end{align*}
  and 
  \begin{align*}
    u_\eps(t)&=u_{0,\eps} + \int_0^tv_\eps(s)\,\d s + 
    \int_0^t(B+\eps G_\eps)(s)\,\d W(s)\,, \\
    \eps v_\eps(t)&+\int_t^Tv_\eps(s)\,\d s + \int_t^T\xi_\eps(s)\,\d s=
    -\eps\int_t^TG_\eps(s)\,\d W(s)\,,
  \end{align*}
  for every $t\in[0,T]$, $\P$-almost surely. In particular, it holds that
  $\partial_t u_\eps^d= v_\eps$, $u_\eps^s=B+\eps G_\eps$, and
  $v_\eps^s=G_\eps$. \medskip

\item[iii)] \emph{(Convergence)} As $\eps\to0$ it holds that 
  \begin{align*}
  u_\eps\wto u \quad&\text{in } L^p_\cP(\Omega; L^p(0,T; V))\cap 
  L^q_\cP(\Omega; W^{s,q}(0,T; V_0^*))
  \quad\forall\,s\in(0,1/2)\,,\\
  v_\eps\wto-\xi \quad&\text{in } L^q_\cP(\Omega; L^q(0,T; V_0^*))\,,\\
  \xi_\eps\wto\xi \quad&\text{in } L^q_\cP(\Omega; L^q(0,T; V^*))\,,\\
  \eps v_\eps\to0 \quad&\text{in } L^q(\Omega; C^0([0,T]; V_0^*)) \cap
  L^2_\cP(\Omega; L^2(0,T; H))\,,\\
  \eps G_\eps\to 0\quad&\text{in } L^q_\cP(\Omega; L^2(0,T; \cL^2(U,V_0^*)))\,,
  \end{align*}
 where $(u,\xi)$ is the unique solution to the problem
 \eqref{prob} in the sense of \eqref{1}-\eqref{3}. 
  Furthermore, if $V\embed H$ compactly and $p<4$, it also holds that 
  \[
  u_\eps \to u \quad\text{in } 
  L^r_\cP(\Omega; L^p(0,T; H)) \qquad\forall\,r\in[1,p)\,.
  \] 
\end{itemize}
\end{thm}

The proof of Theorem \ref{thm:00} is recorded in the coming Sections
\ref{sec:thm1}-\ref{sec:equiv}. In particular, Part ii of the theorem
is proved in Section \ref{sec:thm1}, where we focus on the
well-posedness of the forward-backward regularized problem
\eqref{prob_eps}. Then, the convergence Part iii of  Theorem
\ref{thm:00} is proved in Section \ref{sec:thm2}. Eventually, the
existence of minimizers is checked in Section \ref{sec:equiv}.

This counterintuitive structuring of the proof of Theorem \ref{thm:00}
is motivated by the fact that the existence of minimizers
of $I_\eps$ follows from proving that the corresponding Euler-Lagrange
problem has a unique solution. One hence has to check the
well-posedness of problem \eqref{prob_eps} first.


\section{The forward-backward regularized problem}
\label{sec:thm1}
This section is devoted to proof of the well-posedness of 
the $\eps$-regularized problem \eqref{prob_eps}
in the sense of Theorem~\ref{thm:00}.ii.
Throughout the section, $\eps>0$ is fixed. 

First of all, let $A_H$ be the random and time-dependent 
unbounded operator on $H$ defined as 
\[
  A_H:\Omega\times[0,T]\times H\to 2^H\,,\qquad
  A_H(\omega,t,z):=A(\omega,t,z)\cap H\,, \quad (\omega,t)\in\Omega\times[0,T]\,,
  \quad z\in H\,.
\]
It is not difficult to show that, for every $(\omega,t)\in\Omega\times[0,T]$, 
the unbounded operator $A_H(\omega,t,\cdot)$ is maximal monotone on $H$.
Indeed, the monotonicity is an immediate consequence of the monotonicity of $A$.
As for the maximality, note that the operator
$I+A(\omega,t,\cdot):V\to V^*$, where $I$ is the identity in $H$
(namely, $\langle I v,w \rangle =(v,w) \ \forall v,w \in H$),
is maximal monotone and coercive by assumption on $A$, hence it is
surjective, which yields the maximality of $A_H(\omega,t,
\cdot)$.
Furthermore, since $A$ is $\cP\otimes\cB(V)/\cB(V^*)$-Effros-measurable, 
it follows that $A_H$ is $\cP\otimes\cB(H)/\cB(H)$-Effros-measurable
as well.

\subsection{The approximation}
Since $A_H$ is maximal monotone on $H$ in its last component, 
for any $\lambda>0$ 
its resolvent and its Yosida approximation are well defined, respectively, as
\[
  J_\lambda:\Omega\times[0,T]\times H\to H\,, \qquad
  J_\lambda(\omega,t,z):=(I+\lambda A_H(\omega,t,\cdot))^{-1}(z)\,, \quad
  (\omega,t,z)\in\Omega\times[0,T]\times H\,,
\]
and
\[
  A_\lambda:\Omega\times[0,T]\times H\to H\,, \qquad
  A_\lambda(\omega,t,z):=\frac{v - J_\lambda(\omega,t,z)}\lambda\,, 
  \quad(\omega,t,z)\in\Omega\times[0,T]\times H\,.
\]
It is well-known that $J_\lambda$ and $A_\lambda$ are
$1$- and $1/\lambda$-Lipschitz-continuous in their third component, respectively, 
uniformly in $\Omega\times[0,T]$.
Moreover, the Effros-measurability of $A_H$ implies that 
$J_\lambda$ and $A_\lambda$ are $\cP\otimes\cB(H)/\cB(H)$-measurable
(see for example \cite[Prop.~3.12]{liu-step}).

For any $\lambda>0$, we consider the approximated problem
\beq
  \label{app}
  \begin{cases}
    
    \d u_{\eps\lambda} = v_{\eps\lambda}\,\d t + (B+\eps G_{\eps\lambda})\,\d W\,,  \\
    u_{\eps\lambda}(0)=u_{0,\eps}\,,
  \end{cases}
  \qquad
  \begin{cases}
    -\eps\d v_{\eps\lambda} + v_{\eps\lambda}\,\d t + A_\lambda(u_{\eps\lambda})\,\d t 
    = - \eps G_{\eps\lambda}\,\d W\,,\\
    \eps v_{\eps\lambda}(T)=0\,.
  \end{cases}
\eeq
We say that a triplet $(u_{\eps\lambda}, v_{\eps\lambda}, G_{\eps\lambda})$ 
is a solution to the approximated problem \eqref{app} if
\[
  (u_{\eps\lambda}, v_{\eps\lambda}, G_{\eps\lambda})  \in L^2(\Omega; C^0([0,T]; H))\times 
  L^2(\Omega; C^0([0,T]; H)) \times 
  L^2_{\cP}(\Omega; L^2(0,T;  \cL^2(U,H))
\]
and it holds that 
\begin{align*}
    u_{\eps\lambda}(t)&=u_{0,\eps} + \int_0^tv_{\eps\lambda}(s)\,\d s 
    +  \int_0^t(B+\eps G_{\eps\lambda})(s)\,\d W(s)\,,   \\
    \eps v_{\eps\lambda}(t)&+\int_t^Tv_{\eps\lambda}(s)\,\d s 
    + \int_t^TA_\lambda(s,u_{\eps\lambda}(s))\,\d s=
    -\eps\int_t^TG_{\eps\lambda}(s)\,\d W(s)\,,
  \end{align*}
  for every $t\in[0,T]$, $\P$-almost surely.

  \subsection{Existence of solutions to the approximated problem}
  We prove here that the approximated problem \eqref{app} admits a
  solution $(u_{\eps\lambda}, v_{\eps\lambda}, G_{\eps\lambda})$.
  To this end,
we characterize the 
  the unique solution $(u_{\eps\lambda}, v_{\eps\lambda}, G_{\eps\lambda})$
  as the unique minimizer of a suitable approximated WED functional.
  
  Let us first introduce some preliminary notation. 
  Note that we have the representation 
\[
  I_\eps= I_\eps^1 + S_\eps +  I_\eps^2\,,
\]
where
  \begin{align*}
  &I_\eps^1:\mathcal I^{2,2}(H,H)\to[0,+\infty)\,,\qquad 
  I_\eps^2:L^p_\cP(\Omega; L^p(0,T; V))\to[0,+\infty)\,,\\
  &I_\eps^1(z):=\E\int_0^Te^{-s/\eps}\left[\frac\eps2\norm{\partial_tz^d(s)}^2_H
  +
  \frac12\norm{(z^s-B)(s)}^2_{\cL^2(U,H)}\right]\,\d s\,,
   \quad z\in \mathcal I^{2,2}(H,H)\,,\\
  &I_\eps^2(z):=\E\int_0^Te^{-s/\eps}\Phi(s,z(s))\,\d s\,, \quad z\in L^p_\cP(\Omega; L^p(0,T; V))\,,
  \end{align*}
  and
  \[
  S_\eps:\mathcal I^{2,2}(H,H)\to[0,+\infty]\,, \qquad
  S_\eps(z):=\begin{cases}
  0 \quad&\text{if } z^d(0)=u_{0,\eps}\,,\\
  +\infty \quad&\text{otherwise}\,.
  \end{cases}
  \]
Moreover, it will be useful to introduce the notation
\[
  \mathcal I^{2,2}_0(H,H):=\left\{h\in \mathcal I^{2,2}(H,H): h^d(0)=0\right\}\,.
\]
The natural candidate as WED functional related to 
the approximated problem \eqref{app} is clearly given by 
In this spirit, we introduce the functional 
\[
  I_{\eps\lambda}:\mathcal I^{2,2}(H,H)\to [0,+\infty)\,, \qquad
  I_{\eps\lambda}:=I^1_\eps + S_\eps + I_{\eps\lambda}^2\,,
\]
with
\begin{align*}
  &I_{\eps\lambda}^2:\mathcal I^{2,2}(H,H)\to[0,+\infty)\,,\\
  &I_{\eps\lambda}(z):=\E\int_0^Te^{-s/\eps}\Phi_\lambda(s,z(s))\,\d s\,,
  \quad z\in\mathcal I^{2,2}(H,H)\,,
\end{align*}
where $\Phi_\lambda$ is the Moreau-Yosida regularisation of $\Phi$, i.e.
\begin{align*}
  &\Phi_\lambda:\Omega\times[0,T]\times H\to [0,+\infty)\,,\\
  &\Phi_\lambda(\omega,t,z):=\Phi(\omega,t,J_\lambda(z)) + 
  \frac1{2\lambda}\norm{z-J_\lambda(z)}_H^2\,,
  \qquad (\omega,t,z)\in\Omega\times[0,T]\times H\,.
\end{align*}
It is well known that $\Phi_\lambda(\omega,t,\cdot)$ is G\^ateaux-differentiable
on $H$ with derivative $A_\lambda(\omega,t,\cdot)$, for every $(\omega,t)\in\Omega\times[0,T]$.

We now show that the approximated problem \eqref{app} is
equivalent to the minimization of $I_{\eps\lambda}$.
In this direction, we aim now at characterizing the subdifferential of $I_{\eps\lambda}$.
This will follow after some intermediate steps.

First of all, we characterize the subdifferential of the sum $I_\eps^1 + S_\eps$.
\begin{lem}
  \label{subdiff1}
  The subdifferential of $I^1_\eps + S_\eps:\mathcal I^{2,2}(H,H)\to[0,+\infty]$ 
  is the operator 
  \[
  \partial(I_\eps^1 + S_\eps):\mathcal I^{2,2}(H,H)
  \to 2^{\mathcal I^{2,2}(H,H)^*}
  \]
  defined in the following way:
  \[
   D(\partial(I_\eps^1+S_\eps)):=\{z\in\mathcal I^{2,2}(H,H): z^d(0)=u_{0,\eps}\}\,,
  \]
  and, for every $z\in D(\partial(I_\eps^1+S_\eps))$ and $w\in \mathcal I^{2,2}(H,H)^*$,
  \[
  w\in\partial(I_\eps^1+S_\eps)(z) 
  \]
  if and only if there exists $\tilde w\in\mathcal I_0^{2,2}(H,H)^\perp$ such that 
  \begin{align*}
  &\ip{w}{h}_{\mathcal I^{2,2}(H,H)}\\
  &\qquad=\E\int_0^Te^{-s/\eps}\left[\eps(\partial_t z^d(s), \partial_t h^d(s))
  +((z^s-B)(s), h^s(s))_{\cL^2(U,H)}\right]\,\d s
  +\ip{\tilde w}{h}_{\mathcal I^{2,2}(H,H)}
  \end{align*}
  for every $h\in\mathcal I^{2,2}(H,H)$.
\end{lem}
\begin{proof}
  First of all, 
  it is clear that $I_\eps^1$ is proper, convex, and lower semicontinuous on
  $\mathcal I^{2,2}(H,H)$. Moreover, 
  we have that $I_\eps^1$ is actually
  G\^ateaux-differentiable. Indeed, 
  for every $z,h\in\mathcal I^{2,2}(H,H)$ and $\delta\neq0$ we have
  \begin{align*}
    \frac{I_\eps^1(z+\delta h)-I_\eps^1(z)}{\delta}&=
    \E\int_0^Te^{-s/\eps}\left[\eps(\partial_tz^d(s), \partial_t
                                                     h^d(s))
    +((z^s-B)(s), h^s(s))_{\cL^2(U,H)}\right]\,\d s\\
    &+\delta\E\int_0^Te^{-s/\eps}\left[\frac\eps2\norm{\partial_t h^d(s)}^2_H
    +\frac12\norm{h^s(s)}_{\cL^2(U,H)}^2\right]\,\d s\,,
  \end{align*}
  where the second term on the right-hand side converges to $0$ as $\delta\to0$
  since $h\in \mathcal I^{2,2}(H,H)$. Hence, $I_\eps^1$ is G\^ateaux-differentiable,
  its G\^ateaux-differential coincides with its subdifferential and it is given by
  \begin{align*}
  &\partial I_\eps^1:\mathcal I^{2,2}(H,H)\to \mathcal I^{2,2}(H,H)^*\,,\\
  &\ip{\partial I_\eps^1(z)}{h}_{\mathcal I^{2,2}(H,H)}=
  \E\int_0^Te^{-s/\eps}\left[\eps(\partial_tz^d(s), \partial_t
    h^d(s))
    +((z^s-B)(s), h^s(s))_{\cL^2(U,H)}\right]\,\d s\,,\\
  &\qquad z,h\in\mathcal I^{2,2}(H,H)\,.
  \end{align*}
  Secondly, $S_\eps$ is proper, convex, and 
  lower semicontinuous on $\mathcal I^{2,2}(H,H)$. Moreover,
  its subdifferential is given by 
  \begin{align*}
  &\partial S_\eps:\mathcal I^{2,2}(H;H)\to 2^{\mathcal I^{2,2}(H,H)^*}\,,\\
  &\partial S_\eps(z):=\left\{w\in\mathcal I^{2,2}(H;H):
  \ip{w}{h}_{\mathcal I^{2,2}(H,H)}=0 \quad\forall\,h\in\mathcal I^{2,2}(H,H), \; h^d(0)=0\right\}\\
  &z\in D(\partial S_\eps):=\left\{z\in I^{2,2}(H,H): z^d(0)=u_{0,\eps}\right\}\,.
  \end{align*}
  In other words, 
  we have that 
  \[
  \partial S_\eps(z)=\mathcal I^{2,2}_0(H,H)^\perp\,, \quad z\in D(\partial S_\eps)\,.
  \]
  Consequently, since $\partial I_\eps^1$ and $\partial S_\eps$ are maximal monotone 
  on $\mathcal I^{2,2}(H,H)$, and 
  \[
  \operatorname{Int}(D(I_\eps^1))\cap D(S_\eps)=
  \mathcal I^{2,2}(H,H)\cap \left\{z\in \mathcal I^{2,2}(H,H): z^d(0)=u_{0,\eps}\right\} \neq\emptyset\,,
  \]
  a classical result on convex analysis (see \cite[Thm.~2.10]{barbu-monot}) ensures that 
  \[
  \partial(I_\eps^1 + S_\eps) = \partial I_\eps^1 + \partial S_\eps\,,
  \]
  with
  \[
  D(\partial(I_\eps^1 + S_\eps)) = D(\partial I_\eps^1)\cap D(\partial S_\eps)=
  D(\partial S_\eps)\,.
  \]
  This implies that, for every $z\in D(\partial(I_\eps^1 + S_\eps))$ and 
  $w\in\partial(I_\eps^1+S_\eps)$,
  we have
  \[
  w=\partial I_\eps^1(z) + \tilde w
  \]
  for a certain $\tilde w\in \mathcal I^{2,2}_0(H,H)^\perp$, as required.
\end{proof}

Now, we characterize the subdifferential of $I_{\eps\lambda}^2$.
\begin{lem}
  \label{subdiff2}
  The subdifferential of $I_{\eps\lambda}^2:\mathcal I^{2,2}(H,H)\to [0,+\infty)$
  is the single-valued operator 
  \begin{align*}
  &\partial I_{\eps\lambda}:\mathcal I^{2,2}(H,H)\to \mathcal I^{2,2}(H,H)^*\,,\\
  &\ip{\partial I_{\eps\lambda}^2(z)}{h}_{\mathcal I^{2,2}(H,H)}
  =\E\int_0^Te^{-s/\eps}(A_\lambda(s,z(s)), h(s))\,\d s\,,
  \qquad z,h\in\mathcal I^{2,2}(H,H)\,.
  \end{align*}
  In particular, it holds that $D(\partial I_{\eps\lambda}^2)=\mathcal I^{2,2}(H,H)$.
\end{lem}
\begin{proof}
  The proof is consequence of a classical computation:
  see for example \cite[Prop.~1.1]{ken_parab}.
\end{proof}

We are now able to characterize 
the subdifferential of the functional $I_{\eps\lambda}$.
\begin{lem}
  \label{subdiff_lam}
  The subdifferential of $I_{\eps\lambda}:\mathcal I^{2,2}(H,H)\to[0,+\infty]$
  is the operator 
  \[
  \partial I_{\eps\lambda}:\mathcal I^{2,2}(H,H)
  \to 2^{\mathcal I^{2,2}(H,H)^*}
  \]
  defined in the following way:
  \[
  D(\partial I_{\eps\lambda}):=\{z\in\mathcal I^{2,2}(H,H): z^d(0)=u_{0,\eps}\}\,,
  \]
  and,
  for every $z\in D(\partial I_{\eps\lambda})$ and $w\in \mathcal I^{2,2}(H,H)^*$,
  \[
  w\in\partial I_{\eps\lambda}(z) 
  \]
  if and only if there exists $\tilde w\in\mathcal I_0^{2,2}(H,H)^\perp$ such that,
  for every $h\in\mathcal I^{2,2}(H,H)$,
  \begin{align*}
  &\ip{w}{h}_{\mathcal I^{2,2}(H,H)}=\E\int_0^Te^{-s/\eps}\left[\eps(\partial_t z^d(s), \partial_t h^d(s))
  +(A_\lambda(s,z(s)), h(s)) \right.\\
  &\qquad\qquad\qquad\qquad\qquad
  \left.+((z^s-B)(s), h^s(s))_{\cL^2(U,H)}\right]\,\d s
  +\ip{\tilde w}{h}_{\mathcal I^{2,2}(H,H)}\,.
  \end{align*}
  In particular, for every $z\in \mathcal I^{2,2}(H,H)$ with $z^d(0)=u_{0,\eps}$
  and $w\in \partial I_{\eps\lambda}(z)$ it holds that
   \begin{align*}
  &\ip{w}{h}_{\mathcal I^{2,2}(H,H)}\\
  &=\E\int_0^Te^{-s/\eps}\left[\eps(\partial_tz^d(s), \partial_t h^d(s))
    +(A_\lambda(s,z(s)), h(s))
    +((z^s-B)(s), h^s(s))_{\cL^2(U,H)}\right]\,\d s 
  \end{align*}
  for every $h\in\mathcal I^{2,2}_0(H,H)$.
\end{lem}
\begin{proof}
  Since $D(I_{\eps}^1+S_\eps)=\{z\in\mathcal I^{2,2}(H,H): z^d(0)=u_{0,\eps}\}$
  and $D(I_{\eps\lambda}^2)=\mathcal I^{2,2}(H,H)$, we have
  \[
  D(I_{\eps}^1+S_\eps)\cap \operatorname{Int}(D(I_{\eps\lambda}^2))=
  \{z\in\mathcal I^{2,2}(H,H): z^d(0)=u_{0,\eps}\}\neq\emptyset\,.
  \]
  Hence, by the classical result \cite[Thm.~2.10]{barbu-monot}, we infer that 
  \[
  \partial I_{\eps\lambda}=\partial (I_{\eps}^1+S_\eps) + \partial I_{\eps\lambda}^2\,,
  \qquad D(\partial I_{\eps\lambda})=D(\partial (I_{\eps}^1+S_\eps)) \cap D(\partial I_{\eps\lambda}^2)\,.
  \]
  The thesis follows then directly from Lemma~\ref{subdiff1} and Lemma~\ref{subdiff2}.
\end{proof}

We have now all the tools in order to 
show existence of solutions to the approximated problem \eqref{app}
via minimization of the regularized functional $I_{\eps\lambda}$.
Namely, we have the following result.
\begin{prop}[Well-posedness of the approximated problem]
  \label{equiv_lam}
  For every $\lambda>0$, the functional $I_{\eps\lambda}$ admits a unique 
  global minimizer 
  \[
  z_{\eps\lambda}\in \mathcal I^{2,2}(H,H)\,.
  \]
  Moreover, the triplet $(z_{\eps\lambda}, \partial_tz_{\eps\lambda}^d, 
  (\partial_t z_{\eps\lambda}^d)^s)$ is a solution 
  of  the approximated problem \eqref{app}.
\end{prop}
\begin{proof}
  We note first that the functional $I_\eps^1+S_\eps$ is strictly convex
  and coercive on $\mathcal I^{2,2}(H,H)$, hence
  so is the functional $I_{\eps\lambda}$ since $\Phi_\lambda$ is
  convex and bounded from below.
  Since $\mathcal I^{2,2}(H,H)$ is reflexive, this ensures the existence and 
  uniqueness of a global minimizer $z_{\eps\lambda}\in\mathcal I^{2,2}(H,H)$ 
  for $I_{\eps\lambda}$. Clearly, we have that 
  $z_{\eps\lambda}\in D(I_{\eps\lambda})$, so that $z_{\eps\lambda}^d(0)=u_{0,\eps}$.
  Moreover, by definition of minimizer 
  we have that 
  \[
  0\in \partial I_{\eps\lambda}(z_{\eps\lambda})\,.
  \]
  By virtue of Lemma~\ref{subdiff_lam}, we deduce that 
  \begin{align}
   & \E\int_0^Te^{-t/\eps}
    \left[\eps(\partial_tz_{\eps\lambda}^d(t), \partial_t h^d(t))
    +(A_\lambda(t,z_{\eps\lambda}(t)), h(t)) ,
    h^s(t))_{\cL^2(U,H)}\right]\,\d t\nonumber\\
&\quad +\E\int_0^Te^{-t/\eps}
    ((z_{\eps\lambda}^s-B)(t), h^s(t))_{\cL^2(U,H)}\,\d t =0 \label{star}
  \end{align}
  for every $h\in\mathcal I^{2,2}(H,H)$ with $h(0)=0$. Now, 
  since $z_{\eps\lambda}\in L^2(\Omega; C^0([0,T]; H))$
  and $A_\lambda$ is uniformly Lipschitz-continuous in its third variable, 
  it is clear that
  \[
  \int_0^\cdot e^{-s/\eps}A_\lambda(s, z_{\eps\lambda}(s))\,\d s
  \in L^2(\Omega; C^1([0,T]; H))\,.
  \]
  Hence, by It\^o's formula we have, in differential (formal) form, that 
  \begin{align*}
  &\d\left(\int_0^t e^{-s/\eps}A_\lambda(s, z_{\eps\lambda}(s))\,\d s, h(t)\right)\\
  &=
  \left(e^{-t/\eps} A(t, z_{\eps\lambda}), h(t)\right)\,\d t
  +\left(\int_0^t e^{-s/\eps}A_\lambda(s, z_{\eps\lambda}(s))\,\d s, \partial_t h^d(t)\right)\,\d t\\
  &\qquad+\left(\int_0^t e^{-r/\eps}
  A_\lambda(r, z_{\eps\lambda}(r))\,\d r, h^s(t)\,\d W(t)\right)\,.
  \end{align*}
  Integrating on $[0,T]$ and taking expectations we infer that 
  \begin{align*}
  \E\left(\int_0^Te^{-s/\eps}A_\lambda(s, z_{\eps\lambda}(s))\,\d s, h(T)\right)&=
  \E\int_0^Te^{-t/\eps}(A(t, z_{\eps\lambda}(t)), h(t))\,\d t\\
  &+\E\int_0^T
  \left(\int_0^t e^{-s/\eps}A_\lambda(s, z_{\eps\lambda}(s))\,\d s, \partial_t h^d(t)\right)\,\d t\,.
  \end{align*}
  Noting that the first term on the right-hand side appears in
  \eqref{star} as well, 
  by substitution we infer then that
  \begin{align}
  \nonumber
  &\E\int_0^T\left(\eps e^{-t/\eps}\partial_t z_{\eps\lambda}^d(t)
  -\int_0^t e^{-s/\eps}A_\lambda(s, z_{\eps\lambda}(s))\,\d s,\partial_th^d(t)\right)\,\d t\\
  \nonumber
  &\quad+\E\int_0^Te^{-t/\eps}
  ((z_{\eps\lambda}^s-B)(t), h^s(t))_{\cL^2(U,H)}\,\d t
  \\
&\quad \label{var_eq}+
\E\left(\int_0^Te^{-s/\eps}A_\lambda(s, z_{\eps\lambda}(s))\,\d s, h(T)\right) =0
  \end{align} 
  for every $h\in\mathcal I^{2,2}(H,H)$ such that $h(0)=0$. Now, note that 
  for any such $h$, we have that 
  \[
  h(T)=h^d(T)+\int_0^Th^s(r)\,\d W(r) = \int_0^T\partial_t h^d(s)\,\d s
  +\int_0^Th^s(r)\,\d W(r) \,,
  \]
  which yields in turn that 
  \begin{align*}
  \E\left(\int_0^Te^{-s/\eps}A_\lambda(s, z_{\eps\lambda}(s))\,\d s, h(T)\right)
  &=\E\int_0^T\left(\int_0^Te^{-r/\eps}A_\lambda(r, z_{\eps\lambda}(r))\,\d r, 
  \partial_t h^d(s)\right)\,\d s\\
  &+\E\left(\int_0^Te^{-s/\eps}A_\lambda(s, z_{\eps\lambda}(s))\,\d s, \int_0^Th^s(r)\,\d W(r)\right)\,.
  \end{align*}
  Using this equality for the last term of
  \eqref{var_eq} we obtain that
  \begin{align}
  \nonumber
  &\E\int_0^T\left(\eps e^{-t/\eps}\partial_t z_{\eps\lambda}^d(t)
  -\int_0^t e^{-s/\eps}A_\lambda(s, z_{\eps\lambda}(s))\,\d r+
  \int_0^Te^{-s/\eps}A_\lambda(s, z_{\eps\lambda}(s))\,\d s
  ,\partial_th^d(t)\right)\,\d t\\
  \label{var_eq2}
  &\quad+\E\left(\int_0^Te^{-t/\eps}(z_{\eps\lambda}^s-B)(t)\,\d W(t)
  +\int_0^Te^{-s/\eps}A_\lambda(s, z_{\eps\lambda}(s))\,\d s,
  \int_0^T h^s(t)\,\d W(t)\right) =0
  \end{align} 
  for every $h\in\mathcal I^{2,2}(H,H)$ with $h^d(0)=0$. 
  Now, for any arbitrary $K\in L^2_{\cP}(\Omega; L^2(0,T; H))$, 
  note that the process
  \[
  h_K:=\int_0^\cdot K(s)\,\d s \in \mathcal I^{2,2}(H,H)
  \]
  satisfies $h_K(0)=0$, hence it is a possible test in
  equation \eqref{var_eq2}.
 Since $h_K^s=0$,
  we deduce that 
  \[
  \E\int_0^T\left(\eps e^{-t/\eps}\partial_t z_{\eps\lambda}^d(t)
  -\int_0^t e^{-s/\eps}A_\lambda(s, z_{\eps\lambda}(s))\,\d s+
  \int_0^Te^{-s/\eps}A_\lambda(s, z_{\eps\lambda}(s))\,\d s,K(t)\right)\,\d t=0
  \]
  for every $K\in L^2_{\cP}(\Omega; L^2(0,T; H))$. Let us stress that the first 
  component of the scalar product appearing in this equality is {\em not}
  progressively measurable, hence one cannot simply deduce that it vanishes 
  by arbitrariness of $K$. Nonetheless, note that by definition of conditional 
  expectation and by the adaptedness of $K$, we have 
  \begin{align*}
  &\E\int_0^T\left(
  \int_0^Te^{-s/\eps}A_\lambda(s, z_{\eps\lambda}(s))\,\d s,K(t)\right)\,\d t\\
  &\qquad=
  \E\int_0^T\left(\E\left[\int_0^Te^{-s/\eps}A_\lambda(s, z_{\eps\lambda}(s))\,\d s\,
  \bigg|\,\cF_t\right],K(t)\right)\,\d t\,.
  \end{align*}
  Since $(\cF_t)_{t\in[0,T]}$ is the filtration generated by $W$ and 
  \[
  \int_0^Te^{-s/\eps}A_\lambda(s, z_{\eps\lambda}(s))\,\d s \in L^2(\Omega,\cF_T; H)\,,
  \]
  the process
  \[ 
  t\mapsto \E\left[\int_0^Te^{-s/\eps}A_\lambda(s, z_{\eps\lambda}(s))\,\d s\,
  \bigg|\,\cF_t\right]
  \]
  is an $H$-valued continuous square-integrable martingale, and in
  particular is 
  progressively measurable. We deduce then that the variational equality 
  reads equivalently 
  \begin{align*}
  &\E\int_0^T\left(\eps e^{-t/\eps}\partial_t z_{\eps\lambda}^d(t)
  {-}\int_0^t e^{-s/\eps}A_\lambda(s, z_{\eps\lambda}(s))\,\d s{+}
  \E\left[\int_0^Te^{-s/\eps}A_\lambda(s, z_{\eps\lambda}(s))\,\d s\,
  \bigg|\cF_t\right],K(t)\right)\,\d t\\
  &=0 \qquad\forall\, K\in L^2_{\cP}(\Omega; L^2(0,T; H))\,.
  \end{align*}
  At this point, since the process appearing on the left term of the scalar product 
  belongs to the space $L^2_{\cP}(\Omega; L^2(0,T; H))$, by arbitrariness of $K$
  we have that 
  \[
  \eps e^{-t/\eps}\partial_t z_{\eps\lambda}^d(t)
  -\int_0^t e^{-s/\eps}A_\lambda(s, z_{\eps\lambda}(s))\,\d s+
  \E\left[\int_0^Te^{-s/\eps}A_\lambda(s, z_{\eps\lambda}(s))\,\d s\,
  \bigg|\,\cF_t\right] = 0
  \]
  almost everywhere in $\Omega\times[0,T]$. We deduce that there is a 
  $\d\P\otimes\d t$-version of 
  $\partial_t z^d_{\eps\lambda}$ (which will be denoted with the same symbol 
  for brevity of notation) such that 
  \beq\label{2nd}
  \partial_t z^d_{\eps\lambda} \in L^2(\Omega; C^0([0,T]; H))\,.
  \eeq
  Furthermore, by the classical martingale representation theorem in Hilbert spaces
  (see e.g.~\cite[Prop.~4.1]{fur-tess} and \cite{hu-peng}), there exists 
  a process $G_{\eps\lambda}\in L^2_\cP(\Omega; L^2(0,T;\cL^2(U,H)))$
  such that 
  \[
  \E\left[\int_0^Te^{-s/\eps}A_\lambda(s, z_{\eps\lambda}(s))\,\d s\,
  \bigg|\,\cF_t\right] = 
  \int_0^Te^{-s/\eps}A_\lambda(s, z_{\eps\lambda}(s))\,\d s +
  \eps\int_t^Te^{-s/\eps}G_{\eps\lambda}(s)\,\d W(s)
  \]
  for every $t\in[0,T]$, from which it follows that 
  \[
  \eps e^{-t/\eps}\partial_t z_{\eps\lambda}^d(t)
  +\int_t^T e^{-s/\eps}A_\lambda(s, z_{\eps\lambda}(s))\,\d s=
  -\eps\int_t^Te^{-s/\eps}G_{\eps\lambda}(s)\,\d W(s) \qquad\forall\,t\in[0,T]\,.
  \]
  It follows in particular that 
  \beq\label{3rd}
  \partial_t z_{\eps\lambda}^d \in \mathcal I^{2,2}(H,H)\,, \qquad \partial_t z^d_{\eps\lambda}(T)=0\,,
  \eeq
  and
  \begin{align*}
  \eps\d(\partial_t z^d_{\eps\lambda})&=
  \eps\d(e^{t/\eps}e^{-t/\eps}\partial_t z^d_{\eps\lambda})=
  e^{t/\eps}(e^{-t/\eps}\partial_t z^d_{\eps\lambda})\,\d t
  +\eps e^{t/\eps}\d(e^{-t/\eps}\partial_t z^d_{\eps\lambda})\\
  &=(\partial_t z^d_{\eps\lambda})\,\d t
  +A_\lambda(\cdot, z_{\eps\lambda})\,\d t
  +\eps G_{\eps\lambda}\,\d W\,,
  \end{align*}
  which reads, equivalently, 
  \beq\label{4th}
  -\eps\d(\partial_t z^d_{\eps\lambda}) + 
  (\partial_t z^d_{\eps\lambda})\,\d t
  +A_\lambda(\cdot, z_{\eps\lambda})\,\d t =
  -\eps G_{\eps\lambda}\,\d W\,.
  \eeq
  
  Now, we go back to the variational formulation \eqref{var_eq2}
  and take $h_L:=L\cdot W\in\mathcal I^{2,2}(H,H)$ as test process
  for any arbitrary $L\in L^2_\cP(\Omega; L^2(0,T; \cL^2(U,H)))$. 
  Clearly, the process $h_L$ satisfies
  $h_L(0)=0$ and is hence a possible test  in \eqref{var_eq2}. Since $\partial_t h^d_L=0$,
  we infer that 
  \[
  \E\left(\int_0^Te^{-t/\eps}(z_{\eps\lambda}^s-B)(t)\,\d W(t)
  +\int_0^Te^{-s/\eps}A_\lambda(s, z_{\eps\lambda}(s))\,\d s,
  (L\cdot W)(T)\right) =0
  \]
  for every $L\in L^2_\cP(\Omega; L^2(0,T; \cL^2(U,H)))$. Now, note that 
  by arbitrariness of $L$, by the usual martingale representation 
  theorems, the random variable $L\cdot W(T)$ is arbitrary in 
  the subspace $L^2_0(\Omega,\cF_T; H)$ of elements with null mean.
  It follows in particular that the precess on the left-term is constant and equal to 
  its mean, i.e.
  \begin{align*}
  &\int_0^Te^{-t/\eps}(z_{\eps\lambda}^s-B)(t)\,\d W(t)
  +\int_0^Te^{-s/\eps}A_\lambda(s, z_{\eps\lambda}(s))\,\d s\\
  &\qquad=
  \E\left[\int_0^Te^{-t/\eps}(z_{\eps\lambda}^s-B)(t)\,\d W(t)
  +\int_0^Te^{-s/\eps}A_\lambda(s, z_{\eps\lambda}(s))\,\d s\right]\\
  &\qquad=
  \E\left[\int_0^Te^{-s/\eps}A_\lambda(s, z_{\eps\lambda}(s))\,\d s\right]\,.
  \end{align*}
  Recalling the definition of $G_{\eps\lambda}$, we have that 
  \[
  \E\left[\int_0^Te^{-s/\eps}A_\lambda(s, z_{\eps\lambda}(s))\,\d s\right] = 
  \int_0^Te^{-s/\eps}A_\lambda(s, z_{\eps\lambda}(s))\,\d s +
  \eps\int_0^Te^{-s/\eps}G_{\eps\lambda}(s)\,\d W(s)\,,
  \]
  so that by comparison we obtain that 
  \[
  \int_0^Te^{-t/\eps}(z_{\eps\lambda}^s-B)(t)\,\d W(t)=
  \eps\int_0^Te^{-s/\eps} G_{\eps\lambda}(s)\,\d W(s) \qquad\P\text{-a.s.}\,,
  \]
  yielding 
  \beq
    \label{1st}
    z^s_{\eps\lambda}=B + \eps G_{\eps\lambda}\,.
  \eeq
   
  It is then clear now from \eqref{2nd}, \eqref{3rd}, \eqref{4th}, and \eqref{1st},
  that $(z_{\eps\lambda}, \partial_tz^d_{\eps\lambda}, G_{\eps\lambda})$
  is a solution to the approximated problem \eqref{app}.
\end{proof}

\subsection{Uniform estimates}
We want to pass now to the limit as $\lambda\searrow0$ in \eqref{app}.
To this end, let us show some uniform estimates in $\lambda$, still with $\eps>0$ fixed.

It\^o's formula for the square of the $H$-norm yields 

\begin{align}
  \nonumber
  \frac12\E\norm{u_{\eps\lambda}(T)}_H^2 &= 
  \frac12\E\norm{u_{0,\eps}}_H^2 + 
  \E\int_0^T\left(v_{\eps\lambda}(s), u_{\eps\lambda}(s)\right)\,\d s
  +\frac12\E\int_0^T\norm{B(s)}^2_{\cL^2(U,H)}\,\d s\\
  \label{ito_aux}
  &+\frac{\eps^2}2\E\int_0^T\norm{G_{\eps\lambda}(s)}^2_{\cL^2(U,H)}\,\d s
  +\eps\E\int_0^T\left(B(s), G_{\eps\lambda}(s)\right)_{\cL^2(U;H)}\,\d s\,.
\end{align}

Note now that 
\[
  \d(\eps v_{\eps\lambda}, u_{\eps\lambda}) = 
  \eps v_{\eps\lambda}\,\d u_{\eps\lambda} + 
  u_{\eps\lambda}\eps\d v_{\eps\lambda} +
  \eps\d[G_{\eps\lambda}, B+\eps G_{\eps\lambda}]\,,
\]
which yields, taking \eqref{app} into account, 
\begin{align*}
  \d(\eps v_{\eps\lambda}, u_{\eps\lambda}) &=
  \eps\norm{v_{\eps\lambda}}_H^2\,\d t + 
  \left(\eps v_{\eps\lambda}, (B+\eps G_{\eps\lambda})\,\d W\right)
  +(v_{\eps\lambda}, u_{\eps\lambda})\,\d t 
  +\left(A_\lambda(u_{\eps\lambda}), u_{\eps\lambda}\right)\,\d t \\
  &+\eps\left(u_{\eps\lambda}, G_{\eps\lambda}\,\d W\right)
  +\eps\left(B,G_{\eps\lambda}\right)_{\cL^2(U,H)}\,\d t
  +\eps^2\norm{G_{\eps\lambda}}^2_{\cL^2(U,H)}\,\d t\,.
\end{align*}
Recalling that $\eps v_\eps(T)=0$, we deduce then that 

\begin{align}
  \nonumber
  &\eps\E\int_0^T\norm{v_{\eps\lambda}(s)}_H^2 \,\d s+ 
  \E\int_0^T(v_{\eps\lambda}(s), u_{\eps\lambda}(s))\,\d s
  +\E\int_0^T(A_\lambda(u_{\eps\lambda}(s)), u_{\eps\lambda}(s))\,\d s\\
  \label{ito_aux2}
  &+\eps^2\E\int_0^T\norm{G_{\eps\lambda}(s)}^2_{\cL^2(U,H)}\,\d s +
  \eps\E\int_0^T\left(B(s), G_{\eps\lambda}(s)\right)_{\cL^2(U;H)}\d s = 
  -\eps\E(v_{\eps\lambda}(0), u_{0,\eps})\,.
\end{align}
Now, noting that
$I - J_\lambda = \lambda A_\lambda$, 
recalling that $A_\lambda(\cdot)\in A(J_\lambda(\cdot))$,
the coercivity of $A$ entails
\begin{align*}
  (A_\lambda(u_{\eps\lambda}), u_{\eps\lambda})&=
  (A_\lambda(u_\lambda), J_\lambda(u_{\eps\lambda}))+
  (A_\lambda(u_{\eps\lambda}), u_{\eps\lambda} - J_\lambda(u_{\eps\lambda}))\\
  &=(A_\lambda(u_\lambda), J_\lambda(u_{\eps\lambda}))+
  \lambda\norm{A_\lambda(u_{\eps\lambda})}_H^2\\
  &\geq c_A\norm{J_\lambda(u_{\eps\lambda})}_V^p + \lambda\norm{A_\lambda(u_{\eps\lambda})}_H^2\,.
\end{align*}
Hence, by comparing \eqref{ito_aux} and \eqref{ito_aux2} we obtain 
\begin{align}
  \nonumber
  &\frac12\E\norm{u_{\eps\lambda}(T)}_H^2 +
  \eps\E\int_0^T\norm{v_{\eps\lambda}(s)}_H^2\,\d s+
  c_A\E\int_0^T\norm{J_\lambda(u_{\eps\lambda}(s))}_V^p\,\d s\\
  \nonumber
  &\qquad+\lambda\E\int_0^T\norm{A_\lambda(u_{\eps\lambda}(s))}_H^2\,\d s 
  +\frac{\eps^2}2\E\int_0^T\norm{G_{\eps\lambda}(s)}^2_{\cL^2(U,H)}\,\d s \\
  \label{ito_aux3}
  &\leq \frac12\E\norm{u_{0,\eps}}_H^2 
  - \eps\E(v_{\eps\lambda}(0), u_{0,\eps})
  + \frac12\E\int_0^T\norm{B(s)}^2_{\cL^2(U;H)}\,\d s\,.
\end{align}

Next, denoting by $R_0:V_0\to V^*_0$ the duality mapping, 
It\^o's formula for the square of the $V^*_0$-norm of $v_{\eps\lambda}$
yields, by \eqref{app},
\begin{align*}
  \frac\eps2\norm{v_{\eps\lambda}(t)}_{V_0^*}^2 &+ 
  \int_{t}^T\norm{v_{\eps\lambda}(s)}_{V_0^*}^2\,\d s 
  +\frac{\eps}2\int_t^T\norm{G_{\eps\lambda}(s)}_{\cL^2(U,V_0^*)}^2\,\d s\\
  &=-\int_t^T\left(A_\lambda(u_{\eps\lambda}(s)), R_0^{-1}(v_{\eps\lambda}(s))\right)\,\d s 
  -\eps\int_t^T\left(R_0^{-1}(v_{\eps\lambda}(s), G_{\eps\lambda}(s)\,\d W(s)\right)
\end{align*}
for every $t\in[0,T]$, $\P$-almost surely.
We would like to write It\^o's formula 
for the $q$-power of the $V_0^*$-norm of $v_{\eps\lambda}$.
Clearly, if $p=2$ then also $q=2$ and nothing has to be done.
If $p>2$ then we have $q\in(1,2)$ and 
this can be achieved by writing It\^o's formula for the 
real function $|\cdot|^{q/2}$.
However, since $q\in(1,2)$ the function 
$|\cdot|^{q/2}$ is not of class $C^2$, and this cannot be done straightaway.
We need then to rely on a suitable approximation of the function $|\cdot|^{q/2}$.
Let us introduce to this end the approximations
\[
  \gamma_\delta:[0,+\infty)\to\erre\,, \qquad \gamma_\delta(r):=(r^2+\delta^2)^{q/4}\,, \quad r\geq0\,, \quad\delta>0\,.
\]
Clearly, we have that $\gamma_\delta\in C^\infty([0,+\infty))$, with 
\[
  \gamma_\delta'(r)=
  \begin{cases}
  \displaystyle\frac{q}2(r^2+\delta^2)^{\frac{q-4}{4}}r \quad&\text{if } r>0\,,\\[2mm]
  0 \quad&\text{if } r=0\,,
  \end{cases}
\]
and
\[
  \gamma_\delta''(r)=
  \begin{cases}
 \displaystyle \frac{q}{2}(r^2+\delta^2)^{\frac{q-4}{4}} + \frac{q(q-4)}4(r^2+\delta^2)^{\frac{q-8}{4}}r^2 \quad&\text{if } r>0\,,\\[2mm]
  0 \quad&\text{if } r=0\,.
  \end{cases}
\]
Consequently, for every $r\geq0$ it holds that
\[
  \lim_{\delta\searrow0}\gamma_\delta(r)=r^{q/2}\,, \qquad
  \lim_{\delta\searrow0}\gamma_\delta'(r)=\frac{q}2r^{q/2-1}1_{\{r>0\}}\,, \qquad
  \lim_{\delta\searrow0}\gamma_\delta''(r)=\frac{q}{2}\frac{q-2}{2}r^{q/2 - 2}1_{\{r>0\}}\,.
\]
Since $\gamma_\delta$ is of class $C^2$, we can use the classical finite dimensional It\^o's 
formula (see e.e.~\cite{dapratozab}) and infer that 
\begin{align*}
  &\gamma_\delta\left(\frac\eps2\norm{v_{\eps\lambda}(t)}_{V_0^*}^2\right) + 
  \int_t^T\gamma_\delta'\left(\frac\eps2\norm{v_{\eps\lambda}(s)}_{V_0^*}^2\right)
  \left(\norm{v_{\eps\lambda}(s)}_{V_0^*}^2+
  \frac\eps2\norm{G_{\eps\lambda}(s)}_{\cL^2(U,V_0^*)}^2\right)\,\d s\\
  &\qquad+\frac{\eps^2}2\int_t^T
  \gamma_\delta''\left(\frac\eps2\norm{v_{\eps\lambda}(s)}_{V_0^*}^2\right)
  \norm{(R_0^{-1}v_{\eps\lambda}(s), G_{\eps\lambda}(s))}_{\cL^2(U,\erre)}^2\,\d s\\
  &=-\int_t^T\gamma_\delta'\left(\frac\eps2\norm{v_{\eps\lambda}(s)}_{V_0^*}^2\right)
  \left(A_\lambda(u_{\eps\lambda}(s)), R_0^{-1}(v_{\eps\lambda}(s))\right)\,\d s \\
  &\qquad-\eps\int_t^T\gamma_\delta'\left(\frac\eps2\norm{v_{\eps\lambda}(s)}_{V_0^*}^2\right)
  \left(R_0^{-1}(v_{\eps\lambda}(s), G_{\eps\lambda}(s)\,\d W(s)\right)\,.
\end{align*}
Now, letting $\delta\searrow0$ it follows by the Dominated Convergence Theorem that 
\begin{align*}
  &\left(\frac{\eps}{2}\right)^{{q}/2}\norm{v_{\eps\lambda}(t)}_{V_0^*}^q + 
  \frac{q}2\left(\frac{\eps}{2}\right)^{{q}/2-1}
  \int_t^T\norm{v_{\eps\lambda}(s)}_{V_0^*}^q\,\d s\\
  &\qquad+\frac{q}2\left(\frac{\eps}{2}\right)^{{q}/2-1}\int_t^T
  \frac\eps21_{\{\norm{v_{\eps\lambda}(s)}_{V_0^*}>0\}}
  \norm{v_{\eps\lambda}(s)}_{V_0^*}^{q-2}
  \norm{G_{\eps\lambda}(s)}_{\cL^2(U,V_0^*)}^2 \,\d s\\
  &\qquad+\frac{q}2\frac{q-2}2\frac{\eps^2}2\left(\frac{\eps}2\right)^{{q}/2-2}
  \int_t^T1_{\{\norm{v_{\eps\lambda}(s)}_{V_0^*}>0\}}
  \norm{v_{\eps\lambda}(s)}_{V^*}^{q-4}
  \norm{(R_0^{-1}v_{\eps\lambda}(s), G_{\eps\lambda}(s))}_{\cL^2(U,\erre)}^2\,\d s\\
  &=-\frac{q}2\left(\frac{\eps}{2}\right)^{{q}/2-1}
  \int_t^T1_{\{\norm{v_{\eps\lambda}(s)}_{V_0^*}>0\}}
  \norm{v_{\eps\lambda}(s)}_{V_0^*}^{q-2}
  \left(A_\lambda(u_{\eps\lambda}(s)), R_0^{-1}(v_{\eps\lambda}(s))\right)\,\d s \\
  &\qquad-\eps\frac{q}2\left(\frac{\eps}{2}\right)^{{q}/2-1}
  \int_t^T1_{\{\norm{v_{\eps\lambda}(s)}_{V_0^*}>0\}}
  \norm{v_{\eps\lambda}(s)}_{V_0^*}^{q-2}
  \left(R_0^{-1}(v_{\eps\lambda}(s), G_{\eps\lambda}(s)\,\d W(s)\right)\,.
\end{align*}
Multiplying by $(\frac\eps2)^{1-\frac q2}$, taking expectations,
and using the Young inequality yields
\begin{align*}
  &\frac{\eps}{2}\E\norm{v_{\eps\lambda}(t)}_{V_0^*}^q + 
  \frac{q}2\E
  \int_t^T\left(\norm{v_{\eps\lambda}(s)}_{V_0^*}^q +
  \frac\eps21_{\{\norm{v_{\eps\lambda}(s)}_{V_0^*}>0\}}
  \norm{v_{\eps\lambda}(s)}_{V_0^*}^{q-2}
  \norm{G_{\eps\lambda}(s)}_{\cL^2(U,V_0^*)}^2 \right)\,\d s\\
  &\qquad+\frac{q(q-2)}4\eps\E
  \int_t^T1_{\{\norm{v_{\eps\lambda}(s)}_{V_0^*}>0\}}
  \norm{v_{\eps\lambda}(s)}_{V_0^*}^{q-4}
  \norm{(R_0^{-1}v_{\eps\lambda}(s), G_{\eps\lambda}(s))}_{\cL^2(U,\erre)}^2\,\d s\\
  &\leq\frac{q}2\E
  \int_t^T\norm{R_0^{-1}v_{\eps\lambda}(s)}_{V}^{q-1}
  \norm{A_\lambda(u_{\eps\lambda}(s))}_{V^*}\,\d s \\
  &\leq\frac{c_0q}{2}\E
  \int_t^T\norm{R_0^{-1}v_{\eps\lambda}(s)}_{V_0}^{q-1}
  \norm{A_\lambda(u_{\eps\lambda}(s))}_{V^*}\,\d s \\
  &\leq\frac q{2p}\E\int_{t}^T\norm{v_{\eps\lambda}(s)}_{V_0^*}^q\,\d s 
  +\frac{c_0^q}2\E\int_t^T\norm{A_\lambda(u_{\eps\lambda}(s))}_{V^*}^q\,\d s\,,
\end{align*}
where $c_0$ denotes the norm of the continuous inclusion $V_0\embed V$.
Since $$\frac q2-\frac q{2p}=\frac q2\left(1-\frac1p\right)=\frac12,$$ by 
rearranging the terms and using the 
boundedness of $A$
we deduce that 
\begin{align}
  \nonumber
  &\frac\eps2\E\norm{v_{\eps\lambda}(t)}_{V_0^*}^q + 
  \frac 12\E\int_{t}^T\norm{v_{\eps\lambda}(s)}_{V_0^*}^q\,\d s \\
  \nonumber
  &\qquad+\eps\frac q4\E\int_t^T1_{\{\norm{v_{\eps\lambda}(s)}_{V_0^*}>0\}}
  \norm{v_{\eps\lambda}(s)}_{V_0^*}^{q-2}
  \norm{G_{\eps\lambda}(s)}_{\cL^2(U,V_0^*)}^2\,\d s\\
  \label{ineq_aux}
  &\leq\frac{c_0^q}{2}\norm{f}_{L^1(\Omega\times(0,T))} + 
  \frac{C_A c_0^q}2\E\int_t^T\norm{J_\lambda(u_{\eps\lambda}(s))}_{V}^p\,\d s
\end{align}
for every $t\in[0,T]$, $\P$-almost surely. 
Now, since $0<q/2<1$, its conjugate exponent 
$-q/(2-q)$ is negative: 
the reverse Young's inequality implies then that 
\begin{align*}
  &\E\int_t^T1_{\{\norm{v_{\eps\lambda}(s)}_{V_0^*}>0\}}
  \norm{v_{\eps\lambda}(s)}_{V_0^*}^{q-2}
  \norm{G_{\eps\lambda}(s)}_{\cL^2(U,V_0^*)}^2\,\d s\\
  &\geq\frac2q\E\int_t^T\norm{G_{\eps\lambda}(s)}_{\cL^2(U,V_0^*)}^q\,\d s
  -\frac{q}{2-q}\E\int_t^T\norm{v_{\eps\lambda}(s)}_{V_0^*}^q\,\d s
\end{align*}
Taking this information into account we deduce from
\eqref{ineq_aux} that
\begin{align*}
  &\frac\eps2\E\norm{v_{\eps\lambda}(t)}_{V_0^*}^q + 
  \frac 12\E\int_{t}^T\norm{v_{\eps\lambda}(s)}_{V_0^*}^q\,\d s 
  +\frac\eps2\E\int_t^T\norm{G_{\eps\lambda}(s)}_{\cL^2(U,V_0^*)}^q\,\d s\\
  &\leq\frac{c_0^q}{2}\norm{f}_{L^1(\Omega\times(0,T))} + 
  \frac{C_A c_0^q}2\E\int_0^T\norm{J_\lambda(u_{\eps\lambda}(s))}_{V}^p\,\d s
  +\frac{q^2}{4(2-q)}\eps\E\int_t^T\norm{v_{\eps\lambda}(s)}_{V_0^*}^q\,\d s\,,
\end{align*}
yielding, by the Gronwall lemma,
\begin{align}
\nonumber
  &\frac\eps2\sup_{t\in[0,T]}\E\norm{v_{\eps\lambda}(t)}_{V_0^*}^q + 
  \frac12\E\int_0^T\norm{v_{\eps\lambda}(s)}_{V_0^*}^q\,\d s 
  +\frac\eps2\E\int_0^T\norm{G_{\eps\lambda}(s)}_{\cL^2(U,V_0^*)}^q\,\d s\\
  \label{ito_aux4}
  &\leq e^{\frac{Tq^2}{2(2-q)}}\left(\frac{c_0^q}{2}\norm{f}_{L^1(\Omega\times(0,T))} + 
  \frac{C_Ac_0^q}2\E\int_0^T\norm{J_\lambda(u_{\eps\lambda}(s))}_{V}^p\,\d s\right)\,.
\end{align}

Now, by multiplying the inequality \eqref{ito_aux4} by 
$e^{-\frac{Tq^2}{2(2-q)}}\frac{c_A}{C_Ac_0^q}$ and
summing it with inequality \eqref{ito_aux3}, the 
last term on the right-hand side of \eqref{ito_aux4} can be incorporated into 
the corresponding term on the left-hand side of \eqref{ito_aux3}:
rearranging the terms, we obtain
\begin{align*}
  &\frac12\E\norm{u_{\eps\lambda}(T)}_H^2\\ 
  &\;+
  \E\int_0^T\left(\eps\norm{v_{\eps\lambda}(s)}_H^2\,\d s+
  \frac{c_A}2\norm{J_\lambda(u_{\eps\lambda}(s))}_V^p
  +\lambda\norm{A_\lambda(u_{\eps\lambda}(s))}_H^2
  +\frac{\eps^2}2\norm{G_{\eps\lambda}(s)}^2_{\cL^2(U,H)}
  \right)\,\d s \\
  &\;+\frac{e^{-\frac{Tq^2}{2(2-q)}}c_A}{2C_Ac_0^q}
  \left(\eps\sup_{t\in[0,T]}\E\norm{v_{\eps\lambda}(t)}_{V_0^*}^q + 
  \E\int_{0}^T\norm{v_{\eps\lambda}(s)}_{V_0^*}^q\,\d s 
  +\eps\E\int_0^T\norm{G_{\eps\lambda}(s)}_{\cL^2(U,V_0^*)}^q\,\d s\right)\\
  &\leq \frac12\E\norm{u_{0,\eps}}_H^2 
  - \eps\E(v_{\eps\lambda}(0), u_{0,\eps})
  + \frac12\E\int_0^T\norm{B(s)}^2_{\cL^2(U;H)}\,\d s
  +\frac{c_A}{2C_A}\norm{f}_{L^1(\Omega\times(0,T))}\,.
\end{align*}
At this point, note the second term 
on the right-hand side above
can be handled using the averaged Young inequality: indeed, we infer that, for every $\sigma>0$,

\begin{align*}
  -\eps\E(v_{\eps\lambda}(0), u_{0,\eps}) 
  &\leq\eps\E\left[\norm{v_{\eps\lambda}(0)}_{V_0^*}\norm{u_{0,\eps}}_{V_0}\right]\\
  &\leq\frac{\sigma^q}{q}\eps\sup_{t\in[0,T]}
  \norm{v_{\eps\lambda}(t)}_{L^q(\Omega,\cF_0; V_0^*)}^q
  +\frac\eps{\sigma^pp}\norm{u_{0,\eps}}_{L^p(\Omega,\cF_0; V_0)}^p \,.
\end{align*}

Choosing and fixing $\sigma$ sufficiently small, independent of $\lambda$ and $\eps$, 
for example
\[
  \sigma:=\left(e^{-\frac{Tq^2}{2(2-q)}}\frac{qc_A}{4C_Ac_0^q}\right)^{1/q}\,,
\]
rearranging the terms we deduce that 
there exists a positive constant 
$M=M(c_A,C_A, c_0, q, T)$, independent of both $\lambda$ and $\eps$, such that
\begin{align*}
  &\E\norm{u_{\eps\lambda}(T)}_H^2 +
  \eps\E\int_0^T\norm{v_{\eps\lambda}(s)}_H^2\,\d s+
  \E\int_0^T\norm{J_\lambda(u_{\eps\lambda}(s))}_V^p\,\d s\\
  &\qquad
  +\lambda\E\int_0^T\norm{A_\lambda(u_{\eps\lambda}(s))}_H^2\,\d s 
  +\eps^2\E\int_0^T\norm{G_{\eps\lambda}(s)}^2_{\cL^2(U,H)}\,\d s\\
  &\qquad+\eps\sup_{t\in[0,T]}\E\norm{v_{\eps\lambda}(t)}_{V_0^*}^q + 
  \E\int_{0}^T\norm{v_{\eps\lambda}(s)}_{V_0^*}^q\,\d s 
  +\eps\E\int_0^T\norm{G_{\eps\lambda}(s)}_{\cL^2(U,V_0^*)}^q\,\d s\\
  &\leq M\E\left(\norm{u_{0,\eps}}_{H}^2 
  + \eps\norm{u_{0,\eps}}_{V_0}^p
  + \norm{B}^2_{L^2(0,T;\cL^2(U;H))}\right)\,.
\end{align*}
At this point, note that by the assumption \eqref{ip_init1}
on $(u_{0,\eps})_\eps$, we have that 
the right-hand side is uniformly bounded in $\eps$ and $\lambda$.

Then, we deduce that, by updating the value of the
constant $M$ (here below and the following possibly changing from
line to line),
\begin{align}
  \label{est1}
  \eps\norm{v_{\eps\lambda}}^2_{L^2(\Omega; L^2(0,T; H))} +
  \norm{J_\lambda(u_{\eps\lambda})}_{L^p(\Omega; L^p(0,T; V))}^p +
  \lambda\norm{A_\lambda(u_{\eps\lambda})}_{L^2(\Omega; L^2(0,T; H))}^2
 & \leq M\,,\\
 \label{est2}
 \eps\norm{v_{\eps\lambda}}^q_{C^0([0,T]; L^q(\Omega; V_0^*))} + 
 \norm{v_{\eps\lambda}}^q_{L^q(\Omega; L^q(0,T; V_0^*))} 
 &\leq M\,,\\
 \label{est2'}
  \eps^2\norm{G_{\eps\lambda}}^2_{L^2(\Omega; L^2(0,T; \cL^2(U,H)))}+
 \eps\norm{G_{\eps\lambda}}^q_{L^q(\Omega; L^q(0,T; \cL^2(U,V_0^*)))}
 &\leq M\,.
\end{align}
In particular, since $(v_{\eps\lambda})$ is uniformly bounded in 
$L^q_{\cP}(\Omega; L^q(0,T; V_0^*))$ by \eqref{est2} and
$(B+\eps G_{\eps\lambda})$ 
is uniformly bounded in $L^2_{\cP}(\Omega; L^2(0,T; \cL^2(U,H)))$
by \eqref{est2'}, 
it follows from the definition of $u_{\eps\lambda}$ itself in \eqref{app} that 
\beq
\label{est3}
  \norm{u_{\eps\lambda}}^q_{L^q(\Omega; C^0([0,T]; V_0^*))}\leq M\,.
\eeq
The boundedness of the operator $A$ yields also 
\beq
  \label{est4}
  \norm{A_\lambda(u_{\eps\lambda})}^q_{L^q(\Omega; L^q(0,T; V^*))}\leq M\,.
\eeq

Furthermore, following a classical argument employed in
backward SPDEs, we can refine the estimate 
on $(v_{\eps\lambda})$.
Indeed, let us recall the already obtained It\^o's formula for 
$v_{\eps\lambda}$ in $V_0^*$, which reads
\begin{align*}
  &\frac{\eps}{2}\norm{v_{\eps\lambda}(t)}_{V_0^*}^q + 
  \frac{q}2
  \int_t^T\left(\norm{v_{\eps\lambda}(s)}_{V_0^*}^q +
  \frac\eps21_{\{\norm{v_{\eps\lambda}(s)}_{V_0^*}>0\}}
  \norm{v_{\eps\lambda}(s)}_{V_0^*}^{q-2}
  \norm{G_{\eps\lambda}(s)}_{\cL^2(U,V_0^*)}^2 \right)\,\d s\\
  &\qquad+\frac{q}2\frac{q-2}2\eps
  \int_t^T1_{\{\norm{v_{\eps\lambda}(s)}_{V_0^*}>0\}}
  \norm{v_{\eps\lambda}(s)}_{V_0^*}^{q-4}
  \norm{(R_0^{-1}v_{\eps\lambda}(s), G_{\eps\lambda}(s))}_{\cL^2(U,\erre)}^2\,\d s\\
  &=-\frac{q}2
  \int_t^T1_{\{\norm{v_{\eps\lambda}(s)}_{V_0^*}>0\}}
  \norm{v_{\eps\lambda}(s)}_{V_0^*}^{q-2}
  \left(A_\lambda(u_{\eps\lambda}(s)), R_0^{-1}(v_{\eps\lambda}(s))\right)\,\d s \\
  &\qquad-\frac{q}2\eps
  \int_t^T1_{\{\norm{v_{\eps\lambda}(s)}_{V_0^*}>0\}}
  \norm{v_{\eps\lambda}(s)}_{V_0^*}^{q-2}
  \left(R_0^{-1}(v_{\eps\lambda}(s), G_{\eps\lambda}(s)\,\d W(s)\right)\,.
\end{align*}
Instead of taking expectations at $t$ fixed, we can now 
take supremum in time and then expectations.
The first term on the right-hand side can be easily bounded 
using the H\"older inequality and the estimates \eqref{est2} and \eqref{est4} as
\begin{align*}
  &\E\int_0^T1_{\{\norm{v_{\eps\lambda}(s)}_{V_0^*}>0\}}
  \norm{v_{\eps\lambda}(s)}_{V_0^*}^{q-2}
  \left(A_\lambda(u_{\eps\lambda}(s)), R_0^{-1}(v_{\eps\lambda}(s))\right)\,\d s\\
  &\leq c_0\E\int_0^T\norm{v_{\eps\lambda}(s)}_{V_0^*}^{q-1}
  \norm{A_\lambda(u_{\eps\lambda}(s))}_{V^*}\,\d s\\
  &\leq c_0
  \norm{v_{\eps\lambda}}_{L^q(\Omega; L^q(0,T; V_0^*))}^{q-1}
  \norm{A_\lambda(u_{\eps\lambda})}_{L^q(\Omega; L^q(0,T; V^*))}\leq M\,.
\end{align*}
The second term on the right-hand side can be bounded, 
thanks to Burkholder-Davis-Gundy and Young inequalities, as
\begin{align*}
  &\E\sup_{t\in[0,T]}\left|\int_t^T1_{\{\norm{v_{\eps\lambda}(s)}_{V_0^*}>0\}}
  \norm{v_{\eps\lambda}(s)}_{V_0^*}^{q-2}
  \left(R_0^{-1}(v_{\eps\lambda}(s), G_{\eps\lambda}(s)\,\d W(s)\right)\right|\\
  & \leq M \E\left(\int_0^T
  1_{\{\norm{v_{\eps\lambda}(s)}_{V_0^*}>0\}}
  \norm{v_{\eps\lambda}(s)}_{V_0^*}^{2(q-2)}
  \norm{\left(R_0^{-1}(v_{\eps\lambda}(s), G_{\eps\lambda}(s)\right)}_{\cL^2(U,\erre)}^2\,\d s 
  \right)^{1/2}\\
  &\leq M \E\left[\norm{v_{\eps\lambda}}_{C^0([0,T]; V_0^*)}^{q/2}
  \!\left(\int_0^T
  1_{\{\norm{v_{\eps\lambda}(s)}_{V_0^*}>0\}}
  \norm{v_{\eps\lambda}(s)}_{V_0^*}^{q-4}
  \norm{\left(R_0^{-1}(v_{\eps\lambda}(s), 
  G_{\eps\lambda}(s)\right)}_{\cL^2(U,\erre)}^2\,\d s\right)^{1/2}\right]\\
  &\leq\sigma\E\norm{v_{\eps\lambda}}^q_{C^0([0,T]; V_0^*)}\\
  &\qquad+\frac{M^2}{4\sigma}\E\int_0^T
  1_{\{\norm{v_{\eps\lambda}(s)}_{V_0^*}>0\}}
  \norm{v_{\eps\lambda}(s)}_{V_0^*}^{q-4}
  \norm{\left(R_0^{-1}(v_{\eps\lambda}(s)), G_{\eps\lambda}(s)\right)}_{\cL^2(U,\erre)}^2\,\d s
\end{align*}
for every $\sigma>0$ (independent of $\lambda$ and $\eps$). Hence, 
choosing $\sigma$ sufficiently small (for example $\sigma:=q/2$), rearranging the terms,
and using the H\"older inequality yields  
\[
  \eps\E\norm{v_{\eps\lambda}}^q_{C^0([0,T]; V_0^*)}\leq 
  M\left(1+\eps\E\int_0^T
  1_{\{\norm{v_{\eps\lambda}(s)}_{V_0^*}>0\}}
  \norm{v_{\eps\lambda}(s)}_{V_0^*}^{q-2}\norm{G_{\eps\lambda}(s)}_{\cL^2(U,V_0^*)}^2\,\d s \right)\,.
\]
Now, note that 
the right-hand side is uniformly bounded in $\lambda$ and $\eps$ thanks to 
the inequality \eqref{ineq_aux} and the already proved estimate
\eqref{est1}. Consequently, we deduce that 
\beq
  \label{est5}
  \eps\norm{v_{\eps\lambda}}^q_{L^q(\Omega; C^0([0,T]; V_0^*))}\leq M\,.
\eeq
Moreover, from inequality \eqref{ineq_aux}, 
since the function $r\mapsto|r|^{q-2}$, $r>0$, is decreasing, 
using again the reverse Young inequality and the estimate \eqref{est1}
we deduce that 
\begin{align*}
  M&\geq \eps\E\int_0^T
  1_{\{\norm{v_{\eps\lambda}(s)}_{V_0^*}>0\}}
  \norm{v_{\eps\lambda}(s)}_{V_0^*}^{q-2}\norm{G_{\eps\lambda}(s)}_{\cL^2(U,V_0^*)}^2\,\d s\\
  &\geq\eps\E\left[\norm{1_{\{\norm{v_{\eps\lambda}}_{V_0^*}>0\}}
  v_{\eps\lambda}}_{C^0([0,T]; V_0^*)}^{q-2}
  \norm{G_{\eps\lambda}}_{L^2(0,T; \cL^2(U,V_0^*))}^2\right]\\
  &\geq \frac2{q}\eps\E\norm{G_{\eps\lambda}}_{L^2(0,T; \cL^2(U,V_0^*))}^q
  -\frac{q}{2-q}\eps\E\norm{v_{\eps\lambda}}^q_{C^0([0,T]; V_0^*)}\,.
\end{align*}
Hence, estimate \eqref{est5} readily implies also 
\beq
  \label{est6}
  \eps\norm{G_{\eps\lambda}}^q_{L^q(\Omega; L^2(0,T; \cL^2(U,V_0^*)))}\leq M\,.
\eeq

\subsection{Passage to the limit as $\lambda\searrow0$}
We pass now to the limit as $\lambda\searrow0$, keeping $\eps>0$ fixed,
and deduce existence of solutions for the regularized problem \eqref{prob_eps_bis}.

The estimates \eqref{est1}--\eqref{est6} imply that there exist
$(u_\eps,\hat u_\eps, v_\eps,\xi_\eps,G_\eps)$
such that, as $\lambda\searrow0$,
\begin{align*}
  u_{\eps\lambda}\wstarto u_\eps \qquad&\text{in } L^q(\Omega; L^\infty(0,T; V_0^*))\,,\\
  J_\lambda(u_{\eps\lambda})\wto \hat u_\eps \qquad&\text{in } L^p(\Omega; L^p(0,T; V))\,,\\
  v_\eps\wstarto v_\eps \qquad&\text{in } L^q(\Omega; L^\infty(0,T; V_0^*))
  \cap L^2(\Omega; L^2(0,T; H))\,,\\
  A_\lambda(u_{\eps\lambda})\wto\xi_\eps \qquad&\text{in } L^q(\Omega; L^q(0,T; V^*))\,,\\  
  G_{\eps\lambda}\wto G_\eps \qquad&\text{in } 
   L^2(\Omega; L^2(0,T; \cL^2(U,H)))\,.
\end{align*}
Note that by the definition of Yosida approximation and estimate \eqref{est1} we have 
\[
  \norm{u_{\eps\lambda}-J_\lambda(u_{\eps\lambda})}_{L^2(\Omega; L^2(0,T; H))}
  =\lambda\norm{A_\lambda(u_{\eps\lambda})}_{L^2(\Omega; L^2(0,T; H))}\leq M\lambda^{1/2}\to 0\,,
\]
which implies  that $\hat u_\eps=u_\eps$. Moreover, by 
letting $\lambda\searrow0$ in the forward equation in \eqref{app}, we get 
\[
  u_\eps=u_{0,\eps} +\int_0^\cdot v_\eps(s)\,\d s + \int_0^\cdot (B+\eps G_\eps)(s)\,\d W(s)\,,
\]
yielding, a posteriori, also that
$u_\eps\in L^2(\Omega; C^0([0,T]; H))$. Similarly, letting $\lambda\searrow0$
in the backward equation in \eqref{app} we obtain, by the weak convergences above, 
\[
  \eps v_\eps + \int_\cdot^Tv_\eps(s)\,\d s + \int_\cdot^T\xi_\eps(s)\,\d s=\eps\int_\cdot^TG_\eps(s)\,\d W(s)\,,
\]
which yields a posteriori that $v_\eps\in L^q(\Omega; C^0([0,T]; V_0^*))$.
Furthermore, by comparison in the equation \eqref{app} it follows in particular that 
\[
  u_{\eps\lambda}(T)\wto u_\eps(T) \quad\text{in } L^2(\Omega,\cF_T; H)\,, \qquad
  v_{\eps\lambda}(0)\wto v_\eps(0) \quad\text{in } L^q(\Omega,
  \cF_0; V_0^*)\,.
\]

It only remains to show that $\xi_\eps\in A(\cdot, u_\eps)$ almost everywhere.
To this end, we recall that by comparison of \eqref{ito_aux} and \eqref{ito_aux2} we have 
\begin{align*}
  &\frac12\E\norm{u_{\eps\lambda}(T)}_H^2 +
  \eps\E\int_0^T\norm{v_{\eps\lambda}(s)}_H^2\,\d s\\
  &\qquad+
  \E\int_0^T(A_\lambda(u_{\eps\lambda}(s)), u_{\eps\lambda}(s))\,\d s
  +\frac{\eps^2}2\E\int_0^T\norm{G_{\eps\lambda}(s)}^2_{\cL^2(U,H)}\,\d s\\
  &\leq \frac12\E\norm{u_{0,\eps}}_H^2 
  - \eps\E(v_{\eps\lambda}(0), u_{0,\eps})
  + \frac12\E\int_0^T\norm{B(s)}^2_{\cL^2(U;H)}\,\d s\,.
\end{align*}
By the weak lower semicontinuity of the norms and
the regularities of the data $B_\eps$ and $u_{0,\eps}$
in condition \eqref{data:eps} we infer then that

\begin{align}
  \nonumber
  &\limsup_{\lambda\searrow0}
  \E\int_0^T(A_\lambda(u_{\eps\lambda}(s)), u_{\eps\lambda}(s))\,\d s\\
  \nonumber
  &=\frac12\E\norm{u_{0,\eps}}_H^2 
  - \eps\E(v_{\eps}(0), u_{0,\eps})
  + \frac12\E\int_0^T\norm{B(s)}^2_{\cL^2(U;H)}\,\d s\\
  \nonumber
  &\qquad
  -\frac12\liminf_{\lambda\searrow0}\E\norm{u_{\eps\lambda}(T)}_H^2
  -\eps\liminf_{\lambda\searrow0}\E\int_0^T\norm{v_{\eps\lambda}(s)}_H^2\,\d s
  -\frac{\eps^2}2
  \liminf_{\lambda\searrow0}\E\int_0^T\norm{G_{\eps\lambda}(s)}_{\cL^2(U,H)}^2\,\d s\\
  \nonumber
  &\leq\frac12\E\norm{u_{0,\eps}}_H^2 
  - \eps\E(v_{\eps}(0), u_{0,\eps})
  + \frac12\E\int_0^T\norm{B(s)}^2_{\cL^2(U;H)}\,\d s\\
  \label{limsup_lam}
  &\qquad
  -\frac12\E\norm{u_{\eps}(T)}_H^2
  -\eps\E\int_0^T\norm{v_{\eps}(s)}_H^2\,\d s
  -\frac{\eps^2}2\E\int_0^T\norm{G_{\eps}(s)}_{\cL^2(U,H)}^2\,\d s\,.
\end{align}

We claim now that the right-hand side of inequality \eqref{limsup_lam} coincides with 
\[
  \E\int_0^T\ip{\xi_\eps(s)}{u_\eps(s)}\,\d s\,.
\]
In order to show this, we replicate in the limit $\lambda=0$
the It\^o's formulas obtained for $\lambda>0$
in \eqref{ito_aux} and \eqref{ito_aux2}. 

Indeed, the It\^o formula for the square of the $H$-norm of $u_\eps$ yields 
\[
  \frac12\E\norm{u_{\eps}(T)}_H^2 = 
  \frac12\E\norm{u_{0,\eps}}_H^2 + \E\int_0^T\left(v_{\eps}(s), u_{\eps}(s)\right)\,\d s
  +\frac12\E\int_0^T\norm{(B+\eps G_\eps)(s)}^2_{\cL^2(U,H)}\,\d s\,,
\]
while It\^o's formula for $(u_\eps, \eps v_\eps)$ yields
\begin{align*}
  &\eps\E\int_0^T\norm{v_\eps(s)}_H^2\,\d s + 
  \E\int_0^T(v_{\eps}(s), u_{\eps}(s))\,\d s
  +\E\int_0^T\ip{\xi_\eps(s)}{u_{\eps}(s)}\,\d s\\
  & = \eps^2\E\int_0^T\norm{G_{\eps}(s)}^2_{\cL^2(U,H)}\,\d s +
  \eps\E\int_0^T\left(B(s), G_{\eps}(s)\right)_{\cL^2(U;H)}\d s
  -\eps\E\ip{v_{\eps}(0)}{u_{0,\eps}} \,.
\end{align*}
By comparison we infer exactly that 
\begin{align}
  \nonumber
  &\frac12\E\norm{u_{\eps}(T)}_H^2 +
  \eps\E\int_0^T\norm{v_\eps(s)}_H^2\,\d s
  +\E\int_0^T\ip{\xi_\eps(s)}{u_{\eps}(s)}\,\d s
  +\frac{\eps^2}2\E\int_0^T\norm{G_{\eps}(s)}^2_{\cL^2(U,H)}\,\d s\\
  &= \frac12\E\norm{u_{0,\eps}}_H^2 
  - \eps\E\ip{v_{\eps}(0)}{u_{0,\eps}}
  \label{ito_eps}
 + \frac12\E\int_0^T\norm{B(s)}^2_{\cL^2(U;H)}\,\d s\,,
\end{align}

as required. Substituting now this expression in the inequality \eqref{limsup_lam}, we get
\[
  \limsup_{\lambda\searrow0}\E\int_0^T(A_\lambda(u_{\eps\lambda}(s)), u_{\eps\lambda(s)})\,\d s
  \leq \E\int_0^T\ip{\xi_\eps(s)}{u_{\eps}(s)}\,\d s\,.
\]
The maximal monotonicity of $A$ implies then that $\xi_\eps\in
A(\cdot, u_\eps)$ almost everywhere, see
\cite[Prop.~2.5, p.~27]{Brezis73}.
Hence, $(u_\eps, \xi_\eps, v_\eps, G_\eps)$ is a solution to \eqref{prob_eps_bis}
in the sense of Theorem~\ref{thm:00}.ii.

\subsection{Uniqueness}
Let us check that the quadruplet $(u_\eps, \xi_\eps, v_\eps, G_\eps)$ is unique.
Assume that $(u_\eps^i, \xi_\eps^i, v_\eps^i, G_\eps^i)$,
for $i=1,2$, solve \eqref{prob_eps_bis} in the sense of
Theorem~\ref{thm:00}.ii.
Then, we have 
\[
  \begin{cases}
  \d (u_\eps^1-u_\eps^2) = 
  (v_\eps^1-v_\eps^2)\,\d t + \eps(G_\eps^1-G_\eps^2)\,\d W  \\
  (u_\eps^1-u_\eps^2)(0)=0\,,
  \end{cases}
\]
and
\[
  \begin{cases}
  -\eps\d(v_\eps^1-v_\eps^2) + (v_\eps^1-v_\eps^2)\,\d t + (\xi_\eps^1-\xi_\eps^2)\,\d t
  =-\eps(G_\eps^1-G_\eps^2)\,\d W\\
  (v_\eps^1-v_\eps^2)(T)=0\,.
  \end{cases}
\]
Using the same argument employed to deduce \eqref{ito_eps},
we infer that 
\begin{align*}
  &\frac12\E\norm{(u_{\eps}^1-u_\eps^2)(T)}_H^2 +
  \eps\E\int_0^T\norm{(v_\eps^1-v_\eps^2)(s)}_H^2\,\d s\\
  &\qquad+\E\int_0^T\ip{(\xi^1_\eps-\xi_\eps^2)(s)}{u_{\eps}(s)}\,\d s
  +\frac{\eps^2}{2}\E\int_0^T\norm{(G_\eps^1-G_\eps^2)(s)}_{\cL^2(U,H)}^2\,\d s
  = 0\,,
\end{align*}
which implies that $v^1_\eps-v_\eps^2=0$ and $G_\eps^1-G_\eps^2=0$ 
by the monotonicity of $A$.
From the forward equation we deduce that $u_\eps^1-u_\eps^2=0$.
By comparison in the the backward equation we find that
$\xi_\eps^1-\xi_\eps^2=0$, as required.
This completes the proof of well-posedness in Theorem~\ref{thm:00}.ii.


\section{The asymptotics as $\eps\searrow0$ of the forward-backward problem}
\label{sec:thm2}
The aim of this section is to show that the solution of the
$\eps$-regularized forward-backward problem \eqref{prob_eps_bis} converges
to the the solution of the nonregularized problem \eqref{eq:0}.

First of all, note that the estimates \eqref{est1}--\eqref{est6} are
independent of $\eps$. Hence, 
by weak lower semicontinuity we deduce that 
\begin{align*}
  \norm{u_\eps}^q_{L^q(\Omega; C^0([0,T]; V_0^*))} + 
  \norm{u_\eps}^p_{L^p(\Omega; L^p(0,T; V))}&\leq M\,,\\
  \eps\norm{v_\eps}^q_{L^q(\Omega; C^0([0,T]; V_0^*))} + 
  \eps\norm{v_\eps}^2_{L^2(\Omega; L^2(0,T; H))}
  +\norm{v_\eps}^q_{L^q(\Omega; L^q(0,T; V_0^*))}&\leq M\,,\\
  \norm{\xi_\eps}^q_{L^q(\Omega; L^q(0,T; V^*))} &\leq M\,,\\
  \eps^2\norm{G_\eps}^2_{L^2(\Omega; L^2(0,T; \cL^2(U,H)))} + 
  \eps\norm{G_\eps}^q_{L^q(\Omega; L^2(0,T; \cL^2(U,V_0^*)))} &\leq M\,.
\end{align*}
Moreover, thanks also to assumption \eqref{ip_init1}
and \cite[Lem.~2.1]{fland-gat}, we have that 
\[
 \norm{u_\eps}_{L^q_\cP(\Omega; W^{s,q}(0,T; V_0^*))}\leq M_s \qquad\forall\,s\in(0,1/2)\,.
\]
We deduce that there exist 
\begin{align*}
  &u\in L^q(\Omega; W^{s,q}(0,T; V_0^*))\cap
  L^q_\cP(\Omega; L^\infty(0,T;V_0^*))\cap L^p_\cP(\Omega; L^p(0,T; V))\,,\\
  &v\in L^q_\cP(\Omega; L^q(0,T; V_0^*))\,, \qquad\xi\in L^q_\cP(\Omega; L^q(0,T; V^*))
\end{align*}
such that, as $\eps\searrow0$,
\begin{align*}
  u_{\eps}\wstarto u \quad&\text{in } L^q(\Omega; W^{s,q}(0,T; V_0^*))\cap
  L^q(\Omega; L^\infty(0,T; H))\cap L^p(\Omega; L^p(0,T; V))\,,\\
  v_{\eps}\wto v \quad&\text{in } L^q(\Omega; L^q(0,T; V_0^*))\,,\\
  \xi_{\eps}\wto \xi \quad&\text{in } L^q(\Omega; L^q(0,T; V^*))\,.
\end{align*}
Moreover, note that 
\begin{align*}
&\norm{\eps v_\eps}_{L^q(\Omega; C^0([0,T]; V_0^*))}=
\eps^{1/p}\eps^{1/q}\norm{v_\eps}_{L^q(\Omega; C^0([0,T]; V_0^*))}
\leq\eps^{1/p}M^{1/q}\to 0\,,\\
&\norm{\eps G_\eps}_{L^q(\Omega; L^2(0,T; \cL^2(U,V_0^*)))}= 
\eps^{1/p}\eps^{1/q}\norm{G_\eps}_{L^q(\Omega; L^2(0,T; \cL^2(U,V_0^*)))}
\leq\eps^{1/p}M^{1/q}\to 0\,,
\end{align*}
so that, by the Burkholder-Davis-Gaundy inequality,
\begin{align*}
  \eps v_\eps\to 0 \quad&\text{in } L^q(\Omega; C^0([0,T]; V_0^*))\,,\\
  \eps G_\eps\to 0 \quad&\text{in } L^q(\Omega; L^2(0,T; \cL^2(U,V_0^*)))\,,\\
  \eps G_\eps\cdot W\to 0 \quad&\text{in } L^q(\Omega; C^0([0,T]; V_0^*))\,.
\end{align*}

Then, by passing to the weak limit as $\eps\searrow0$
in the backward equation in \eqref{prob_eps_bis} yields
\[
  \int_t^Tv(s)\,\d s + \int_t^T \xi(s)\,\d s= 0 \qquad\forall\,t\in[0,T]\,,\quad\P\text{-a.s.}\,,
\]
from which $v+\xi = 0$ almost everywhere in $\Omega\times(0,T)$.
In particular, we have that
\[
  v=-\xi \in L^q_{\cP}(\Omega; L^q(0,T; V^*))\,.
\]
Furthermore,
recalling the convergences \eqref{ip_init1} on the data
and passing to the weak limit in the forward equation in \eqref{prob_eps_bis} we have that
\[
  u = u_0 + \int_0^\cdot v(s)\,\d s + \int_0^\cdot B(s)\,\d W(s)\,.
\]
In particular, since $v\in L^q_\cP(\Omega; L^q(0,T; V^*))$ and 
$u\in L^p_\cP(\Omega; L^p(0,T; V))$, by the classical It\^o's formula
(see \cite{LiuRo})  we deduce
by comparison that $u\in L^2(\Omega; C^0([0,T]; H))$,
while from \eqref{ito_eps} we have
\[
  u_\eps(T)\wto u(T) \quad\text{in } L^2(\Omega; H)\,.
\]
 
Eventually, let us show that $\xi \in A(\cdot, u)$ almost
everywhere. This follows again by lower-semicontinuity
arguments. In particular, from \eqref{ito_eps}, the 
weak lower semicontinuity of the norms, and the convergences \eqref{ip_init1} we have that 
\begin{align*}
  \nonumber
  &\limsup_{\eps\searrow0}
  \E\int_0^T\ip{\xi_\eps(s)}{u_{\eps}(s)}\,\d s\\
  &\leq-\frac12\liminf_{\eps\searrow0}
  \E\norm{u_{\eps}(T)}_H^2
  +\frac12\E\norm{u_0}_H^2 
  + \frac12\E\int_0^T\norm{B(s)}^2_{\cL^2(U;H)}\,\d s
  -\limsup_{\eps\to0} 
  \eps\E(v_{\eps}(0), u_{0,\eps})\\
  &\leq -\frac12\E\norm{u(T)}_H^2
  +\frac12\E\norm{u_0}_H^2 
  + \frac12\E\int_0^T\norm{B(s)}^2_{\cL^2(U;H)}\,\d s\\
  &\qquad+\limsup_{\eps\searrow0}
  \eps\norm{v_\eps(0)}_{L^q(\Omega;V_0^*)}\norm{u_{0,\eps}}_{L^p(\Omega;V_0)}\,.
\end{align*}
The last term on the right-hand side can be handled using 
the estimates above and the condition \eqref{ip_init1} as
\[
  \eps\norm{v_\eps(0)}_{L^q(\Omega;V_0^*)}\norm{u_{0,\eps}}_{L^p(\Omega;V_0)}
  \leq M^{1/q}\eps^{1/p}\norm{u_{0,\eps}}_{L^p(\Omega;V_0)} \to 0\,.
\]
Hence, we infer that 
\[
  \limsup_{\eps\searrow0}
  \E\int_0^T\ip{\xi_\eps(s)}{u_{\eps}(s)}\,\d s\leq
  -\frac12\E\norm{u(T)}_H^2
  +\frac12\E\norm{u_0}_H^2 
  + \frac12\E\int_0^T\norm{B(s)}^2_{\cL^2(U;H)}\,\d s\,.
\]
Now, since by the It\^o formula for $u$ and the fact that $\xi=-v$ we know that 
\[
  \frac12\E\norm{u(T)}_H^2 + \E\int_0^T\ip{\xi(s)}{u(s)}\,\d s=
  \frac12\E\norm{u_0}_H^2 
  + \frac12\E\int_0^T\norm{B(s)}^2_{\cL^2(U;H)}\,\d s\,,
\]
we obtain
\[
  \limsup_{\eps\searrow0}
  \E\int_0^T\ip{\xi_\eps(s)}{u_{\eps}(s)}\,\d s\leq
  \E\int_0^T\ip{\xi(s)}{u(s)}\,\d s\,.
\]
This yields $\xi\in A(\cdot, u)$ almost everywhere in $\Omega\times(0,T)$, 
and the first part of Theorem~\ref{thm:00}.iii is proved.

We only need to show the strong convergences in the 
last assertion of Theorem~\ref{thm:00}.iii,
under the extra assumption that $V\embed H$ is compact and $p<4$.
The idea is to use 
the following classical result by {\sc Gy\"ongy \& Krylov} \cite[Lem.~1.1]{gyo-kry}.
\begin{lem}\label{gy}
  Let $\mathcal X$ be a Polish space and $(Z_n)_n$ be a sequence
  of $\mathcal X$-valued random variables. Then $(Z_n)_n$ converges
  in probability if and only if
  for any pair of subsequences $(Z_{n_k})_k$ and $(Z_{n_j})_j$, there exists 
  a joint sub-subsequence $(Z_{n_{k_i}}, Z_{n_{j_i}})_i$ converging 
  in law to a probability measure $\nu$ on $\mathcal X\times\mathcal X$ such
  that $\nu(\{(z_1,z_2)\in\mathcal X\times\mathcal X: z_1=z_2\})=1$.
\end{lem}

Let then $(u_{\eps_k})_k$ and $(u_{\eps_j})_j$ be arbitrary subsequences of $(u_\eps)$.
By the compactness result \cite[Cor.~5, p.~86]{simon} we have the compact inclusion
\[
  L^p(0,T; V)\cap W^{s,q}(0,T; V_0^*)\cembed L^p(0,T; H)
\]
provided that $s>1/q-1/p=1-2/p$. Since $p<4$ by assumption, 
an easy computation shows that $1-2/p<1/2$: 
hence there exists $\bar s\in(0,1/2)$ such that the compact inclusion holds.
Now, from the estimates we know that 
\[
  \norm{u_\eps}_{L^p(\Omega; L^p(0,T; V)\cap W^{\bar s,q}(0,T; V_0^*))}\leq M\,,
\]
which implies, using a standard argument based on the Markov inequality, that the family
of laws of $(u_{\eps})_\eps$ on $L^p(0,T; H)$ is tight.
By the Skorokhod theorem \cite[Thm.~2.7]{ike-wata},
there exists a probability space $(\Omega',\cF',\P')$
and measurable functions $\phi_\eps:(\Omega',\cF')\to(\Omega,\cF)$
such that $\P'\circ\phi_\eps^{-1}=\P$ for all $\eps>0$ and 
\begin{align*}
  u'_{\eps_{k_i}}:=u_{\eps_{k_i}}\circ\phi_{\eps_{k_i}}
  \to u'_1 \quad\text{in } L^p(0,T; H)\,,\quad\P'\text{-a.s.}\,,\\
  u'_{\eps_{j_i}}:=u_{\eps_{j_i}}\circ\phi_{\eps_{j_i}}
  \to u'_2 \quad\text{in } L^p(0,T; H)\,,\quad\P'\text{-a.s.}\,.
\end{align*}
Relying on the uniform estimates proved above and on the uniqueness of the 
limit problem, it is not difficult to show that 
\[
  \P'\{u_1'(t)=u_2'(t)\;\forall\,t\in[0,T]\}=1\,,
\]
which is exactly the condition of Lemma~\ref{gy}. For further details 
we refer for example to \cite[\S~5]{vall-zimm} and \cite{ScSt-DNL}. 
Hence, the Lemma implies that, on the original probability space, we have 
\[
  u_\eps\to u\quad\text{in } L^p(0,T; H)\,,\quad\P\text{-a.s.}
\]
As $(u_\eps)_\eps$ is bounded in $L^p(\Omega; L^p(0,T; H))$, this yields
\[
  u_\eps\to u \quad\text{in } L^r(\Omega; L^p(0,T; H)) \qquad\forall\,r\in[1,p)\,.
\]
This completes the proof of Theorem~\ref{thm:00}.iii.


\section{Equivalence between regularized equation and minimization problem}
\label{sec:equiv}

This section is devoted to check that $I_\eps$ admits a unique minimizer
in ${\mathcal V}$, and that this coincides with the unique solution to the
$\eps$-regularized problem. This proves
Theorem \ref{thm:00}.i.
In all of this section
$\eps>0$ is kept fixed.

A natural idea would be to identify the subdifferential 
of $I_\eps$ in terms of $\partial (I_\eps^1+S_\eps)$ and $\partial I^2_\eps$.
However, let us point out that the domain of $I_\eps^2$,
i.e.~the space $L^p_\cP(\Omega; L^p(0,T; V))$, may have 
empty interior in the topology of $\mathcal I^{2,2}(H,H)$.
For this reason, the analogous of \cite[Thm.~2.10]{barbu-monot}
is not applicable in this case, and we need to rely again 
on a further approximation, 
obtained by replacing $\Phi$ with its 
Moreau-Yosida approximation $\Phi_\lambda$,
for $\lambda>0$. 

We follow the following strategy instead.
First of all, we show that the unique solution $u_\eps$ to problem \eqref{prob_eps}
is a minimizer for $I_\eps$. This ensures in particular that 
$I_\eps$ admits at least a minimizer. Secondly, 
we note that actually $I_\eps$ admits at most one minimizer.
This eventually entails that minimizing $I_\eps$ is equivalent 
to solving \eqref{prob_eps}.

\begin{prop}
  \label{sol_min}
  The unique solution $u_\eps$ to \eqref{prob_eps} is a minimiser for $I_\eps$.
\end{prop}
\begin{proof}
  From Section~\ref{sec:thm1} we know that $u_\eps$ can be constructed as 
  limit in suitable topologies of a sequence $(u_{\eps\lambda})_{\lambda>0}$, 
  where $u_{\eps\lambda}$ is the unique first solution component of \eqref{app}.
  By Proposition~\ref{equiv_lam} we also know that such $u_{\eps\lambda}$
  is the unique global minimizer of $I_{\eps\lambda}$ for all $\lambda>0$, so that 
  \beq\label{min_lam}
  I_{\eps\lambda}(u_{\eps\lambda})\leq I_{\eps\lambda}(z) \qquad\forall\,z\in\mathcal I^{2,2}(H,H)\,.
  \eeq
  Let us now consider $z\in D(I_\eps)=\mathcal V$:  since $\Phi_\lambda\leq\Phi$,
  we immediately have
  \[
  I_{\eps\lambda}(z)\leq I_\eps(z) \qquad\forall\,z\in\mathcal V\,.
  \]
  Furthermore, by Section~\ref{sec:thm1} we know that 
  \begin{align*}
  J_\lambda(u_{\eps\lambda})\wto u_\eps \quad&\text{in } L^p(\Omega; L^p(0,T; V))\,,\\
  v_{\eps\lambda}\wto v_\eps\quad&\text{in } L^2(\Omega; L^2(0,T; H))\,.
  \end{align*}
  Hence, by the definition of $\Phi_\lambda$, the weak lower semicontinuity of $\Phi$,
  and by the Fatou lemma, we have
  \begin{align*}
  \E\int_0^Te^{-t/\eps}\Phi(t,u_\eps(t))\,\d t&\leq 
  \liminf_{\lambda\searrow0}\E\int_0^Te^{-t/\eps}\Phi(t,J_\lambda(u_{\eps\lambda}(t)))\,\d t\\
  &\leq \liminf_{\lambda\searrow0}\E\int_0^Te^{-t/\eps}\Phi_\lambda(t,u_{\eps\lambda}(t))\,\d t
  \end{align*}
  and
  \[
  \E\int_0^Te^{-t/\eps}\frac\eps2\norm{v_\eps(t)}^2\,\d t\leq 
  \liminf_{\lambda\searrow0}
  \E\int_0^Te^{-t/\eps}\frac\eps2\norm{v_{\eps\lambda}(t)}^2\,\d t\,.
  \]
  Taking these remarks into account, and 
  recalling that $v_{\eps\lambda}=\partial_t u_{\eps\lambda}^d$,
  $v_\eps=\partial_t u_\eps^d$, and $u_{\eps\lambda}^s=u_{\eps}^s=B_\eps$, we have 
  \begin{align*}
  I_\eps(u_\eps)&=\E\int_0^Te^{-t/\eps}\left[\frac\eps2\norm{v_\eps(t)}^2
  +\Phi(t,u_\eps(t))\right]\,\d t\\
  &\leq\liminf_{\lambda\searrow0}
  \E\int_0^Te^{-t/\eps}\left[\frac\eps2\norm{v_{\eps\lambda}(t)}^2
  +\Phi_\lambda(t,u_{\eps\lambda}(t))\right]\,\d t=
  \liminf_{\lambda\searrow0}I_{\eps\lambda}(u_{\eps\lambda})\,.
  \end{align*}
  Passing then to the $\liminf$ in \eqref{min_lam} yields then
  \[
  I_\eps(u_\eps)\leq I_\eps(z) \qquad\forall\,z\in\mathcal V\,,
  \]
  hence $u_\eps$ is a global mininimizer of $I_\eps$, as required.
\end{proof}

In order to conclude the proof of Theorem~\ref{thm:00}.i, 
note that the functional $I_\eps^1+S_\eps$ is 
strictly convex and coercive on $\mathcal I^{2,2}(H,H)$, 
hence so is $I_{\eps}$ on $\mathcal V$ since $\Phi$ is convex and
bounded from below.
Since $\mathcal V$ is reflexive, 
we deduce that $I_{\eps}$ admits a unique global minimizer $z_\eps\in\mathcal V$.
Moreover, by virtue of Proposition~\ref{sol_min}, we know that 
the unique solution $u_\eps\in\mathcal U_{reg}$ to \eqref{prob_eps}
is a global minimizer of $I_\eps$. By uniqueness of $z_\eps$, 
we infer that $z_\eps=u_\eps\in\mathcal U_{reg}$.
This concludes the proof of Theorem~\ref{thm:00}.i.


\section*{Acknowledgement}
LS is supported by the Austrian Science Fund (FWF) project M\,2876. 
US
is supported by the FWF projects F\,65, W\,1245, I\,4354, and
P\,32788 and by the Vienna
Science and Technology Fund (WWTF) project MA14-009.



\def\cprime{$'$}

\end{document}